\documentclass[english]{smfart}
\usepackage{epsfig,   enumerate, psfrag, amsmath,amssymb}
\usepackage[utf8]{inputenc}

\let\mathbb\mathbf

\date{\today}

\keywords{}
\author{Romain Dujardin}
\thanks{Research  partially supported by ANR project LAMBDA,  ANR-13-BS01-0002 and  a grant from the  Institut Universitaire de France}
\makeindex
\makeglossary
 
\title[Arithmetic problems in rational dynamics]{Some problems of arithmetic origin in rational dynamics}

\address{Sorbonne Université, CNRS, Laboratoire de Probabilités, Statistiques et Modélisations (LPSM, UMR 8001), F-75005 Paris, France}
\email{romain.dujardin@upmc.fr}

\subjclass{37P05; 37P30; 37F10; 37F45}

\newcommand{\cc}{\mathbb{C}}

\newcommand{\re}{\mathbb{R}}
\newcommand{\dd}{\mathbb{D}}
\newcommand{\zz}{\mathbb{Z}}
\newcommand{\nn}{\mathbb{N}}
\newcommand{\pp}{\mathbb{P}}
\newcommand{\Q}{\mathbb{Q}}
\renewcommand{\aa}{\mathbb{A}}

\newcommand{\e}{\varepsilon}

\newcommand{\fr}{\partial}
\newcommand{\om}{\Omega}
\newcommand{\set}[1]{\left\{#1\right\}}
\newcommand{\norm}[1]{\left\Vert#1\right\Vert}
\newcommand{\abs}[1]{\left\vert#1\right\vert}
\newcommand{\cd}{{\cc^2}}

\newcommand{\pu}{{\mathbb{P}^1}}

\newcommand{\rest}[1]{ \arrowvert_{#1}}

\newcommand{\unsur}[1]{\frac{1}{#1}}

\newcommand{\cst}{{C}^{st}}
\newcommand{\lrpar}[1]{\left(#1\right)}

\newcommand{\la}{\lambda}
\newcommand{\lo}{{\lambda_0}}

\newcommand{\La}{\Lambda}

\newcommand{\poly}{\mathcal{P}}
\newcommand{\mpoly}{\mathcal{MP}}
\newcommand{\pcf}{post-critically finite }
\newcommand{\cm}{{\mathrm{cm}}}

\newcommand{\inv}{^{-1}}

\DeclareMathOperator{\supp}{Supp}

\DeclareMathOperator{\bif}{Bif}
\DeclareMathOperator{\stab}{Stab}

\DeclareMathOperator{\percrit}{PerCrit}
\DeclareMathOperator{\per}{Per}
 \DeclareMathOperator{\preper}{Preper}

\newtheorem{prop}{Proposition} [section]
\newtheorem{thm}[prop] {Theorem}

\newtheorem{cor}[prop]{Corollary}

 \newtheorem*{thm*}{Theorem}
\newtheorem*{cor*}{Corollary}

\theoremstyle{definition}
\newtheorem{defi}[prop] {Definition}

\theoremstyle{remark}

\makeindex
\makeglossary

\begin{document}
 \maketitle
\begin{abstract}
 These are lecture notes from a course in arithmetic dynamics given in Grenoble in June 2017. 
  The main purpose of this text is to 
   explain how arithmetic equidistribution theory can be used in the dynamics of rational maps on $\mathbf P^1$. 
   We first briefly introduce   the  
 basics of the  iteration theory of rational maps  on the projective line over $\cc$, as well as some elements of iteration theory over an arbitrary complete valued field and  the construction of dynamically 
 defined height functions for rational functions defined over $\overline\Q$. The equidistribution of small points gives 
some original information on the distribution of preperiodic orbits, leading to some non-trivial rigidity statements. We then 
explain some  consequences of arithmetic equidistribution 
to the study of the geometry of parameter spaces 
  of such dynamical systems, notably pertaining to 
  the  distribution of special parameters and the classification of special subvarieties. 
   \end{abstract}

 \tableofcontents
 
 \section*{Introduction}  In the recent years a number of classical  ideas and problems in arithmetics
  have been transposed to the setting of rational 
 dynamics in one and several variables.  A main source of motivation in these developments is the analogy between 
 torsion points on an Abelian variety and (pre-)periodic points of rational maps. This is actually more than an analogy since torsion points on an Abelian variety $A$ are precisely the preperiodic points of the endomorphism of $A$ induced 
  by    multiplication by 2. Thus, problems about the distribution or structure of torsion points can be 
  translated to dynamical problems. This analogy also applies to spaces of such objects:  in this way 
  elliptic curves with complex multiplication would correspond to post-critically finite rational maps. Again one may ask whether 
  results about the distribution of these ``special points'' do reflect each other. This point of view was in particular put forward 
 by J. Silverman (see \cite{silverman, silverman_notes} for a detailed presentation and references).   
  
Our goal is to present a few recent results belonging to  this line  of research. 
More precisely we will concentrate  on  some results in which potential theory and arithmetic equidistribution, as 
presented in this volume by P. Autissier and A. Chambert-Loir (see \cite{autissier, acl}), play a key role. This includes:
\begin{itemize}
\item an arithmetic proof of the equidistribution of periodic orbits towards the equilibrium 
measure as well as some consequences  (Section~\ref{sec:dynamics_equidist}); 
\item the equidistribution of post-critically finite mappings  in the parameter space of degree 2 (Section~\ref{sec:quadratic}) and degree $d\geq 3$ polynomials (Section~\ref{sec:equidist higher});
\item  the classification of special curves in the space of cubic polynomials (Section~\ref{sec:special}). 
\end{itemize}
A large part of these results is based on the work of Baker-DeMarco \cite{baker_demarco1, baker_demarco}
and Favre-Gauthier \cite{favre_gauthier, favre_gauthier2}.

These notes are based on a series of lectures given by the author in a summer school in Grenoble in June 2017, which
were intended for an audience  with minimal knowledge in complex analysis and  dynamical systems. 
 The  style is deliberately informal and favors reading flow against precision, in order 
 to arrive rather quickly at some recent advanced topics. 
 In particular the proofs are mostly sketched, with an emphasis on the  dynamical parts of the  arguments.
The material in Part I is standard and covered with much greater detail in  classical textbooks:
see e.g. Milnor \cite{milnor} or Carleson-Gamelin \cite{carleson_gamelin} for holomorphic dynamics, and Silverman \cite{silverman} 
for the arithmetic side. Silverman's lecture notes \cite{silverman_notes} and DeMarco's 2018 ICM address \cite{demarco_icm}
contain similar but more advanced material.

 \part{Basic holomorphic and arithmetic  dynamics on $\pu$}  
 \section{A few useful geometric tools}

 \Subsection*{Uniformization}
 
 \begin{thm*}[Uniformization Theorem]  
 Every simply connected Riemann surface is biholomorphic to 
 the open unit disk $\dd$, the complex plane $\cc$ or the Riemann sphere $\pu(\cc)$. 
 \end{thm*}

 See \cite{saint_gervais} for a beautiful and  thorough treatment of this result and of its historical context.
A Riemann surface $S$ is called {\em hyperbolic}\index{Hyperbolic>Riemann surface} (resp. {\em parabolic})\index{Parabolic>Riemann surface} if its universal cover is the unit disk (resp. the 
complex plane). Note that in this terminology, an elliptic curve $E\simeq \cc/\Lambda$ is parabolic. 
In a sense, ``generic'' Riemann surfaces are hyperbolic, however, interesting complex dynamics occurs only on 
parabolic Riemann surfaces or on  $\pu(\cc)$. 

\begin{thm*}  If $a$, $b$, $c$ are distinct points on $\pu(\cc)$, then $\pu(\cc)\setminus \set{a,b,c}$ is hyperbolic.
\end{thm*}

\begin{proof}
Since the Möbius group acts transitively on triples of points, we may assume that 
$\set{a,b,c} = \set{0, 1, \infty}$. Fix a base point $\star\in \cc\setminus\set{0, 1}$, then the fundamental group
$\pi_1(  \cc\setminus\set{0, 1}, \star)$ is free on two generators. 
Let
$S\to \cc\setminus\set{0, 1}$ be a universal cover. Since $\cc\setminus\set{0, 1}$ is non-compact, then $S$ is biholomorphic to
 $\dd$ or
$\cc$. The deck transformation group is a group of automorphisms of $S$ acting discretely and isomorphic 
to  the free group on two generators. The affine group $\mathrm{Aut}(\cc)$ has no free subgroups (since for instance 
its   commutator subgroup is Abelian) so necessarily $S\simeq\dd$. 
\end{proof}

\begin{cor*} If $\Omega$ is an open subset of $\pu(\cc)$ whose complement contains at least 3 points, then $\Omega$ is hyperbolic.\end{cor*}

Recall that a family $\mathcal{F}$
of meromorphic functions on some open set $\Omega$ is said {\em normal}\index{Normal family} if it is equicontinuous 
w.r.t. the spherical metric, or equivalently, if it is relatively compact in the compact-open topology. Concretely, if $(f_n)$ is
a sequence in a normal family $\mathcal{F}$, then
 there exists a subsequence converging locally uniformly on $\Omega$ to 
a meromorphic function. 
 
\begin{thm*}[Montel]
If $\Omega$ is an open subset of $\pu(\cc)$ and $\mathcal{F}$ is a family of meromorphic functions in $\Omega$  avoiding
3 values, then $\mathcal{F}$ is normal.
\end{thm*}

\begin{proof}
By post-composing with a Möbius transformation
 we may assume that $\mathcal{F}$ avoids $\set{0, 1, \infty}$. 
Pick a disk $D\subset \Omega$ and a sequence $(f_n)\in \mathcal{F}^\mathbb{N}$. 
It is enough to show that $(f_n\rest{D})$ is normal. A first possibility is that $(f_n\rest{D})$ diverges uniformly in
 $\cc\setminus \set{0,1}$, that is is converges to  $\set{0, 1, \infty}$, and in this case we
  are done. Otherwise there exists $t_0\in D$, 
 $z_0\in \Omega$ and a subsequence $(n_j)$ such that $f_{n_j}(t_0)\to z_0$. Then if $\psi: \dd\to 
  \cc\setminus\set{0, 1}$ is a universal cover such that $\psi(0)  = z_0$, lifting 
  $(f_{n_j})$ under $\psi$  yields a sequence  $\hat f_{n_j}: D\to \dd$ such that 
  $\hat f_{n_j}(t_0)\to 0$. 
The Cauchy estimates imply that 
$(\hat f_{n_j})$ is locally uniformly Lipschitz in $D$ so  by   the Ascoli-Arzela theorem, extracting further if necessary, 
 $(\hat f_{n_j})$ converges to some $\hat f: D\to \dd$. Finally, $f_{n_j}$ converges to $\psi\circ \hat f$ and we are done. 
\end{proof}

\subsection*{The hyperbolic metric}

First recall the {\bf Schwarz-Lemma}: {\em if $f:\dd\to\dd$ is a holomorphic map such that $f(0)=0$, then $\abs{f'(0)}\leq 1$, with equality if and only if $f$ is an automorphism, which must then  be a rotation.} 

More generally, any automorphism of $\dd$ is of the form $f(z) = e^{i\theta} \frac{z-\alpha}{1-\bar\alpha z}$, for some 
$\theta\in \re$ and 
$\abs{\alpha}<1$, and the inequality 
in the Schwarz Lemma can then be propagated as follows: let $\rho$ be the Riemannian metric on $\dd$ defined by the formula $\rho(z) = \frac{2\abs{dz}}{1-\abs{z}^2}$, i.e. if $v$ is a tangent vector at $z$, $v\in T_z\dd\simeq \cc$, 
then $\norm{v}_\rho  = \frac{2\abs{v}}{1-\abs{z}^2}$.  This   metric is referred to as the {\em hyperbolic (or Poincaré) metric}\index{Hyperbolic>metric}\index{Poincar\'e>metric}\index{Metric>hyperbolic=(Hyperbolic ---)}\index{Metric>poincare=(Poincar\'e ---)} and 
the Riemannian manifold $(\dd, \rho)$ is the {\em hyperbolic disk}\index{Hyperbolic>disk}. 

\begin{thm*}[Schwarz-Pick Lemma]\index{Schwarz--Pick>lemma}
Any holomorphic map $f:\dd\to \dd$ is a weak contraction for the hyperbolic metric. It is a strict contraction unless $f$ is an automorphism.
\end{thm*}

If now $S$ is any hyperbolic Riemann surface,  since the deck transformation group of the universal cover 
$\dd\to  S$ acts by isometries, we can push $\rho$ to a well-defined  {\em hyperbolic metric} on $S$. As before any holomorphic
map $f:S\to S'$ 
between hyperbolic Riemann surfaces is a weak contraction, and it is a strict contraction unless   $f$ lifts to an isometry between their 
universal covers.  

 \section{Review of rational dynamics on  $\pu(\cc)$}
 
 Let $f:\pu(\cc)\to \pu(\cc)$ be a rational map of degree $d$, that is 
 $f$ can be written
  in homogeneous coordinates as $[P(z,w):Q(z,w)]$, where $P$ and $Q$ are homogeneous polynomials of degree $d$ 
 without common factors. Equivalently, $f(z) = \frac{P(z)}{Q(z)}$ in some affine chart. It is an elementary fact
  that  any holomorphic self-map on $\pu(\cc)$ is rational. In particular the group of automorphisms of $\pu(\cc)$  is
  the M\"obius group $\mathrm{PGL}(2, \cc)$.
   
 We consider $f$ as a dynamical system, that is we wish to understand the asymptotic behavior of the iterates 
 $f\circ \cdots \circ f =: f^n$. General references for the results of this section include \cite{milnor, carleson_gamelin}.
 Throughout these notes, we make the standing assumption that $d\geq 2$. 
 
 \subsection*{Fatou-Julia dichotomy}\index{Fatou--Julia>dichotomy}  
The {\em Fatou set}\index{Fatou set} $F(f)$\glossary{$F(f)$: Fatou set} is the  set of points $z\in \pu(\cc)$ such that there exists a neighborhood $N\ni z$ on which the sequence 
of iterates $(f^n\rest{N})$ is normal. It is open by definition.  A typical situation occurring in the Fatou set 
is that of an attracting orbit: that is an invariant  finite set $A$   such that every  nearby point
 converges under iteration to $A$. 
The locus of non-normality or  {\em Julia set}\index{Julia set} $J(f)$\glossary{$J(f)$: Julia set} is   the  complement of the Fatou set: 
$J(f) = F(f)^\complement$. When there is no danger of confusion we feel free to drop the dependence on $f$.
 
The Fatou and Julia sets are invariant ($f(X) \subset X$, where $X = F(f)$ or $J(f)$) and even totally invariant ($f\inv(X) = X$).  
 Since  $d\geq 2$  it   easily follows that $J(f)$ is nonempty. On the other hand the Fatou set may be empty.

 If $z\in J(f)$ and $N$ is a neighborhood of $z$, it follows from Montel's theorem that 
 $\bigcup_{n\geq 0} f^n(N)$ avoids at most 2 points. Define 
 $$E_z  = \bigcap_{N\ni z} \lrpar{ \pu(\cc) \setminus \bigcup_{n\geq 0} f^n(N)}.$$ 
 
\begin{prop}
The set  $E:= E_z$ is independent of $z\in J(f)$. Its cardinality is  at most 2 and it 
is the maximal totally invariant finite subset. 
  Furthermore:
 \begin{itemize}
 \item If $\#E=1$ then  $f$ is conjugate in $\mathrm{PGL}(2, \cc)$ to a  polynomial.
 \item If $\#E =2$ then   $f$ is conjugate in $\mathrm{PGL}(2, \cc)$  to  $z\mapsto z^{\pm d}$. 
 \end{itemize}
 \end{prop}
 
 The set $E(f)$\glossary{$E(f)$: Exceptional set} is called the {\em exceptional set}\index{Exceptional set} of $f$. 
 Note that it is always an attracting periodic orbit, in particular it is contained in the Fatou set. 
 
 \begin{proof} 
It is immediate that
 $f^{-1}(E_z) \subset  E_z$ and by Montel's Theorem $\#E_z \leq 2$.    If $\#E_z= 1$, then we may conjugate so that 
 $E_z = \set{\infty}$ and since $f^{-1}(\infty) = \infty$ it follows that $f$ is a polynomial. If $\#E_z= 2$ we conjugate so that 
 $E_z = \set{0, \infty}$. If both points are fixed then it is easy to see that $f(z) = z^d$, and in the last case that $f(z)  = z^{-d}$. 
 The independence with respect to $z$ follows easily. 
 \end{proof}

The following is an immediate consequence of the definition of $E$. 
 
 \begin{cor}\label{cor:exceptional}
 For every $z\notin E$, $\displaystyle \overline {\bigcup\nolimits_{n\geq 0} f^{-n}(z) }\supset J$. 
 \end{cor}

Let us also note the following consequence of Montel's Theorem. 
 
 \begin{cor}
If the Julia set $J$ has non-empty interior,   then $J  =\pu(\cc)$. 
 \end{cor}

 \subsection*{What does $J(f)$ look like?}
The Julia set is closed, invariant, and infinite (apply  
  Corollary \ref{cor:exceptional} to $z\in J$). 
 It can be shown that $J$ is perfect (i.e. has no isolated points).

  A first possibility is that {\bf $J$ is the whole sphere} $\pu(\cc)$. It is a  deep result by M. Rees \cite{rees}   
  that this
   occurs with positive probability   
 if a rational map is chosen at random in the  space of all rational maps of a given degree. 
   Note that this never happens for polynomials because $\infty$ is an attracting point. 
  
  Important explicit examples of rational maps with $J =\pu$ are {\em Lattès mappings},\index{Lattesmapping=Latt\`es mapping} which are rational maps coming from 
  multiplication on an elliptic curve. In a nutshell: let $E$ be an elliptic curve, viewed as a torus  
  $  \cc/ \Lambda$, where $\Lambda$ is a lattice, and let $m:\cc\to\cc$ 
  be a $\cc$-linear 
  map such that $m(\Lambda)\subset \Lambda$. Then $m$ commutes with the involution $z\mapsto -z$, so it descends 
  to a self-map of  the 
  quotient Riemann surface $E/(z\sim (-z))$ which turns out to be $\pu(\cc)$. The calculations can be worked out
   explicitly  for the doubling map
   using elementary properties of the  Weierstra\ss~$\wp$-function, and one finds for instance that 
   $f(z) =  \frac{(z^2+1)^2}{4z(z^2-1)}$ is a Lattès example (see e.g. \cite{milnor_lattes} for this and more about Lattès mappings).

  It can also happen that {\bf $J$ is a smooth curve}  (say of class $C^1$). A classical theorem of Fatou then asserts that $J$ must be contained in a circle (this includes lines in $\cc$), and more precisely:
  \begin{itemize}
 \item  either it is a circle: this happens for $z\mapsto z^{\pm d}$ but also for some Blaschke products;
 \item or it is an interval in a circle: this happens for interval maps such as the  Chebychev polynomial $z\mapsto 
 z^2-2$ for which $J = [-2, 2]$. 
 \end{itemize}

 Otherwise  {\bf $J$ is a self-similar ``fractal'' set}\index{Fractal set} with often complicated topological structure (see Figure \ref{fig:julia}). 
  
  \begin{figure}
    \includegraphics[width=\textwidth]{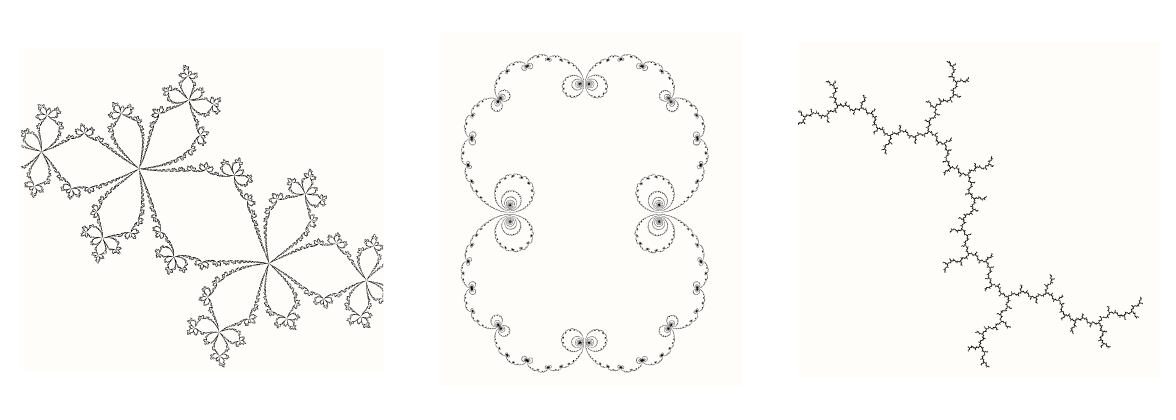}
  \caption{Gallery of quadratic Julia sets $J(f_c)$ with $f_c(z) = z^2+c$. 
  Left: $z^2+ (-.5+.556i)$, middle: $z^2+.2531$, right: $z^2+i$.}\label{fig:julia}
  \end{figure}

 \subsection*{Periodic points} A point $z\in \pu(\cc)$ is \emph{periodic} if there exists $n$ such that $f^n(z)   = z$. The period of 
 $z$ is the minimal such positive $n$,  And a fixed point is a point of period 1. 
 Elementary algebra shows that $f^n$ admits $d^n+1$ fixed points counting multiplicities. If $z_0$ has exact period $n$, the {\em multiplier}\index{Multiplier>of a periodic point} of $z_0$ is the complex number $(f^n)'(z_0)$ (which does not depend on the chosen Riemannian  metric on the sphere). It determines a lot of the local dynamics of $f^n$ near $z_0$.  There are three main cases:
 \begin{itemize}
 \item {\em attracting:}\index{Attracting periodic point}\index{Periodic point>attracting=(Attracting ---)} $\abs{(f^n)'(z_0)}<1$: then $z_0\in F$ and for   $z$ near $z_0$, $f^{nk}(z)\to z_0$ as $k\to \infty$;
 \item {\em repelling:}\index{Repelling periodic point}\index{Periodic point>repelling=(Repelling ---)} $\abs{(f^n)'(z_0)}>1$: then $z_0\in J$ since $\abs{(f^{nk})'(z_0)}\to  \infty$ as $k\to \infty$;
 \item {\em neutral:}\index{Neutral periodic point}\index{Periodic point>neutral=(Neutral ---)} $\abs{(f^n)'(z_0)}=1$: then $z_0$ belongs to either  $F$ or $J$.
 \end{itemize}
 The neutral case subdivides further into {\em rationally neutral} (or {\em parabolic}) when $(f^n)'(z_0)$ is a root of unity and 
 {\em irrationally neutral} otherwise. Parabolic points belong to the Julia set. The 
 dynamical classification of irrationally neutral points, which essentially boils down to the question whether they belong to 
 the Fatou or Julia set,  is quite delicate (and actually still not complete) and we will not need it. 
 
 \begin{thm}[Fatou, Julia]
 Repelling periodic points are dense in $J(f)$.
 \end{thm}
 
 This is an explanation for the local self-similarity of $J$. In particular we   see that if 
  $z_0$ is repelling with a non-real multiplier, then
 $J$ has a ``spiralling structure'' at $z_0$, in particular it cannot be smooth.  A delicate result by Eremenko and Van Strien \cite{eremenko_vanstrien} asserts conversely that  if
 all periodic points multipliers are real, then $J$ is contained in a circle. 
 
 \begin{proof}[Idea of proof]
 The first observation  is that if $g$ is a holomorphic  self-map of an open topological disk $\Delta$ such that $\overline{g(\Delta)}
 \subset \Delta$, then by contraction of the Poincaré metric, 
 $g$ must have an attracting fixed point. If now $z_0\in J$ is arbitrary we know from Montel's Theorem and the classification 
 of exceptional points that for any neighborhood
 $N \ni z_0$, $\bigcup_{n\geq 0} f^n(N)$ eventually covers $J$, hence $z_0$. With some more work, it can be  ensured  that there exists a small open disk $D$ close to  $z_0$ and an integer $n$ such that $f^n$ is univalent on $D$ and $\overline D \subset f^n (D)$. Then the result follows from the initial observation. 
 \end{proof} 
 
 In particular a rational map admits infinitely many repelling points. It turns out that
 conversely the number of non-repelling periodic points is finite. The first step is the following basic result. 
 
 \begin{thm}[Fatou]\label{thm:attracting critical}
 Every attracting periodic orbit attracts a critical point. 
 \end{thm}
 
By this mean that for every attracting periodic orbit $A$ there exists a critical point $c$ such that 
$f^n(c)$ tends to $A$ as $n\to\infty$. 
 This is a basic instance of a general heuristic principle: the dynamics is 
 determined by the behavior of critical points. 
 
 \begin{proof}
 Let $z_0$ be an attracting point of period $n$. For expositional ease we replace $f$ by $f^n$  so assume $n=1$. Let 
 $$\mathcal{B} = \set{z\in \pu, \ f^k(z)\underset{k\to \infty}\longrightarrow z_0}$$ be the {\em basin}\index{Basin>of an attracting point} of $z_0$ 
 and $\mathcal{B}_0$ be the {\em immediate basin}, that is the connected component of $z_0$ in $\mathcal{B}$. Then $f$ maps 
 $\mathcal{B}_0$ into itself and $\mathcal{B}_0$ is hyperbolic since $\mathcal{B}_0\cap J = \emptyset$.   
 If $f\rest{\mathcal{B}_0}$ had no critical points,  then $f\rest{\mathcal{B}_0}$ would be a covering, hence a local 
 isometry for the hyperbolic metric. This contradicts the fact that $z_0$ is attracting. 
   \end{proof}

 Since $f$ has $2d-2$ critical points we infer:
 
 \begin{cor}
 A rational map of degree $d$ admits at most $2d-2$ attracting periodic orbits. 
 \end{cor}

  \begin{thm} [Fatou, Shishikura] \label{thm:fatou_shishikura}
 A rational map   admits  only finitely many non-repelling periodic orbits.
 \end{thm}
 
  \begin{proof}[Idea of proof]
  The bound $6d-6$ was first obtained by Fatou 
  by a beautiful perturbative argument. First, the previous theorem (and its corollary) can be 
easily extended to the case of periodic orbits with multiplier equal to 1 (these   are attracting 
 in a certain direction). So there are at most $2d-2$ periodic points with multiplier 1. Let now $N$ be the number of neutral periodic 
 points with multiplier different from 1. Fatou shows that under a generic perturbation of $f$, 
 at least half of them become attracting. Hence $N/2\leq 2d-2$ and the result follows. 
 
 The sharp bound $2d-2$ for the total number of non-repelling cycles was obtained by Shishikura. 
 \end{proof}
 
 \subsection*{Fatou dynamics}
 The dynamics in the Fatou set can be understood completely. Albeit natural,  this is far from obvious: one can
 easily imagine an open set $U$  on which the   iterates form a normal family, yet the asymptotic behavior of  
 $(f^n\rest{U})$ is complicated to analyze. 
Such a phenomenon  actually happens   
  in transcendental or   higher dimensional   dynamics, but not in our setting: 
 the key point is the following deep and celebrated ``non-wandering domain theorem''.

 \begin{thm}[Sullivan]\index{Sullivan>non-wandering domain theorem}
 Every component $U$ of the Fatou set
  is ultimately periodic, i.e. there exist   integers $ l>k$   such that $f^l(U) = f^k(U)$. 
  \end{thm}
 
 It remains to classify periodic components. 
 
 \begin{thm}
 Every  component $U$ of period $k$ of the Fatou set is 
 \begin{itemize}
 \item either an attraction domain:   as $n\to\infty$, 
  $(f^{nk})$ converges locally uniformly on $U$ to   a periodic cycle (attracting or parabolic);
 \item or a rotation domain: $U$ is biholomorphic to a disk or an annulus  and $f^k\rest{U}$ is holomorphically conjugate to an 
 irrational rotation.
 \end{itemize}
 \end{thm}
 
 \section{Equilibrium measure}
In many arithmetic  applications the  most important dynamical object is the {\em equilibrium measure}.\index{Equilibrium measure}\index{Measure>equilibrium=(Equilibrium ---)} The following theorem 
summarizes its construction and main properties. 

\begin{thm}[Brolin, Lyubich, Freire-Lopes-Mañé] \label{thm:equilibrium} 
Let $f$ be a rational map   on $\pu(\cc)$ of degree $d\geq 2$. Then for any $z\notin E(f)$, the sequence of probability measures $$\mu_{n,z}  = \unsur{d^n} \sum_{w\in f^{-n} (z)} \delta_w$$ (where pre-images are counted with their multiplicity) 
converges weakly to a Borel probability measure $\mu = 
\mu_f$ on $\pu(\cc)$ 
enjoying the following properties:
\begin{enumerate}[(i)]
\item $\mu$ is invariant, that is  $f_*\mu = \mu$, and has the ``constant Jacobian'' property $f^*\mu = d \mu$;
\item $\mu$  is ergodic; 
\item $\supp(\mu) = J(f) $;
\item $\mu$ is repelling, that is for $\mu$-a.e. $z$, $\displaystyle \liminf_{n\to\infty} \unsur{n}\log\norm{df^n_z}\geq \frac{\log d}{2}$, 
where the norm of the differential is computed with respect to any smooth Riemannian metric on $\pu(\cc)$;
\item $\mu$ describes the asymptotic distribution of repelling periodic orbits, that is, if $\mathrm{RPer}_n$ denotes the set of repelling periodic points of exact period $n$, then $\unsur{d^n}\sum_{z\in \mathrm{RPer}_n} \delta_z$ converges to $\mu$ as $n\to \infty$. 
\end{enumerate}
\end{thm}

Recall that if $\mu$ is a probability measure, its image $f_*\mu$ under $f$  is defined by 
$f_*\mu(A) = \mu(f\inv(A))$ for any Borel set $A$. The pull-back $f^*\mu$ is conveniently defined by its action on continuous functions: 
$\langle f^*\mu , \varphi\rangle  = \langle  \mu, f_*\varphi\rangle$, where $f_*\varphi$ is defined by 
$f_*\varphi(x) = \sum_{y\in f\inv(x)} \varphi(y)$.  
Recall also that ergodicity\index{Ergodicity} means that any measurable invariant subset has measure 0 or 1.  
A stronger form of ergodicity actually holds: $\mu$ is {\em exact}, that is, if $A$ is any measurable subset of
 positive measure, then 
$\mu(f^n(A))\to 1$ as $n\to\infty$. 

The proof of this theorem is too long to be explained in these notes (see \cite{guedj} for a detailed treatment). 
Let us only discuss the convergence statement.
 
\begin{proof}[Proof of the convergence of the $\mu_{n,z}$]
Consider the Fubini-Study $(1,1)$ form $\omega$ associated to the spherical metric on $\pu$. It expresses in coordinates as 
$\omega= \frac{i}{\pi} \fr\overline \fr\log \norm{\sigma}^2$ where $\sigma: \pu\to \cd\setminus\set{0}$ is any local section of the canonical projection
$\cd\setminus\set{0}\to \pu$  and $\norm{\cdot}$ is the standard Hermitian norm on $\cd$, i.e 
$\norm{(z,w)}^2 = \abs{z}^2+\abs{w}^2$. Then 
$$  f^*\omega -  d \omega =   \frac{i}{\pi} \fr\overline \fr\log \frac{\norm{\lrpar{P(z,w), Q(z,w)}}^2}{\norm{(z,w)}^{2d}}  = :
\frac{i}{\pi} \fr\overline \fr g_0,$$
where by homogeneity of the polynomials $P$ and $Q$,  $g_0$ is a globally well-defined smooth function on $\pu$. Thus 
we infer that 
\begin{align*}
\unsur{d^n} (f^n)^*\omega - \omega &=
\sum_{k=0}^{n-1} \lrpar{\unsur{d^{k+1}}  (f^{k+1})^*\omega  - \unsur{d^{k }}
(f^{k })^*\omega } \\
&= \sum_{k=0}^{n-1}\unsur{d^{k+1}}  (f^{k})^* (f^*\omega - d\omega) \\
&=  \frac{i}{\pi} \fr\overline \fr\lrpar{ \sum_{k=0}^{n-1}  \unsur{d^{k+1}} g_0\circ f^k} =  : \frac{i}{\pi} \fr\overline \fr g_k.
\end{align*}
One readily sees that the  sequence of functions $g_k$ converges uniformly to a continuous function $g_\infty$, 
therefore basic distributional calculus implies that 
$$ \unsur{d^n} (f^n)^*\omega \underset{n\to\infty}\longrightarrow \omega + \frac{i}{\pi} \fr\overline \fr g_\infty. $$
Being a limit in the sense of distributions of a sequence of probability measures, 
this last term can be identified to a  positive measure $\mu$ by declaring that 
$$\langle\mu, \varphi\rangle =  \int_{\pu} \varphi \,\omega + \frac{i}{\pi} \int_\pi g_\infty\, \fr\overline \fr \varphi$$ for any smooth function 
$\varphi$. By definition this measure is the equilibrium measure $\mu_f$. 

Now for any $z\in \pu$,  identifying
(1,1) forms and signed measures as above
 we write 
$$ \unsur{d^n} \sum_{w\in f^{-n} (z)} \delta_w  = \unsur{d^n} (f^n)^*\delta_z = \unsur{d^n} (f^n)^*\omega +  \unsur{d^n} (f^n)^*(\delta_z - \omega).$$ There exists a $L^1$ function on $\pu$ such that in the sense of distributions 
$ \frac{i}{\pi} \fr\overline \fr g_z = \delta_z - \omega$, so that  
$$\unsur{d^n} (f^n)^*(\delta_z - \omega)  = \frac{i}{\pi} \fr\overline \fr \lrpar{ \unsur{d^n} g_z \circ f^n} $$
and to establish the convergence of the $\mu_{n,z}$ the problem is to show that for $z\notin E$, $ {d^{-n} } g_z \circ f^n$ converges to 0 in $L^1$. It is rather easy to prove that  this convergence holds for a.e. $z$ with respect to Lebesgue measure. Indeed, suppose that the  $g_z$ are chosen with some uniformity, for instance by assuming $\sup_{\pu}{g_z}  = 0$. In this case 
the average $\int g_z  \mathrm{Leb}(dz)$ is a bounded function $g$ (such that $\frac{i}{\pi} \fr\overline \fr g = \mathrm{Leb}$)
and the result essentially follows from  the Borel-Cantelli Lemma: indeed, on average,  ${d^{-n}} g_z \circ f^n$ is bounded by 
$\cst/d^n$. 

Proving the convergence for every $z\notin E$ requires a finer analysis, see \cite{guedj} for details. 
\end{proof}

 \subsection*{The   case of polynomials}
If $f$ is a polynomial, rather than    the Fubini-Study measure, we can use the Dirac mass at the totally invariant point 
$\infty$, to give another  formulation of  these results. Indeed, if $\nu$ is any probability measure 
(say with compact support) in $\cc$,   write $\nu - \delta_\infty = \frac{i}{\pi} \fr\overline \fr g$ as before. In $\cc$, this   rewrites as 
 $\nu = \Delta g$, where\footnote{For convenience we have swallowed the normalization constant in $\Delta$, so that $\Delta$ is $1/4\pi$ times the ordinary Laplacian.} $g$ is a subharmonic function with logarithmic growth at infinity: 
 $g(z) = \log\abs{z} + c+ o(1)$. If    $\nu_{S^1}$ be the normalized Lebesgue measure  on the unit circle, then 
 $\nu_{S^1}  = \Delta\lrpar{ \log^+\abs{z}}$ (where $\log^+t = \max(\log t, 0)$), so 
 $$\unsur{d^n}(f^n)^* (\nu_{S^1}) =    \unsur{d^n} \Delta\lrpar{ \log^+\abs{f^n(z)}}. $$ An argument similar to that of the proof of Theorem  \ref{thm:equilibrium} yields the following:
 
 \begin{prop}
 If $f$ is a polynomial of degree $d$, 
 the sequence of functions $d^{-n} \log^+ \abs{f^n}$ converges locally uniformly to a subharmonic function 
 $G :\cc\to \re$ satisfying $G  \circ f = d G$. 
 \end{prop}

 The function $G = G_f$ is by definition the {\em dynamical Green function}\index{Dynamical Green function}\index{Green>dynamicalfunction=(Dynamical --- function)} of $f$.  
  Introduce the {\em filled Julia set}\index{Julia set>filled=( filled ---)}\index{Filled Julia set}
\begin{align*}
K(f) &= \set{z\in \cc, \ (f^n(z)) \text{ is bounded in } \cc} \\
&=  \set{z\in \cc, \ (f^n(z)) \text{ does not tend to }\infty}.
\end{align*} \glossary{$K(f)$: Filled julia set}
  It is easy to show that $J(f)  = \fr K(f)$.
  
   \begin{prop} The dynamical Green function  has the following
  properties:
  \begin{enumerate}[(i)] 
  \item $G_f$ is continuous, non-negative and subharmonic in $\cc$;
  \item $\set{G_f =0} = K(f)$;
  \item $\Delta G_f = \mu_f$ is the equilibrium measure (in particular $\supp(\Delta G_f)= \fr \set{G_f =0}$);
  \item if $f$ is a monic polynomial, then $G_f(z) = \log\abs{z} +o(1)$ as $z\to\infty$. 
 \end{enumerate}
 \end{prop}
 
Properties {\em (i)-(iii)} show that the dynamically defined function  $G_f$ 
coincides with the Green function of $K(f)$ from classical potential theory. Property {\em (iv)} implies that if $f$ is monic, then 
$K(f)$ is of capacity 1, which is important in view of its number-theoretic properties (see the lectures by P. Autissier in this volume \cite{autissier}). 

Another useful consequence is that the Green function 
$G_f$ is completely determined  by the compact set $K(f)$ (or by its boundary $J(f)$). 
In particular if $f$ and $g$ are polynomials of degree at least 2 then 
\begin{equation}\label{eq:equivalence_julia}
J(f)   = J(g) \Leftrightarrow \mu_f = \mu_g. 
\end{equation}
The same is not true for rational maps: $J(f)$ does not determine $\mu_f$ in general. For instance there are many rational maps 
$f$ such that  $J = \pu(\cc)$, however  
Zdunik \cite{zdunik} proved that  $\mu_f$ is absolutely continuous with respect to  the Lebesgue measure on 
$\pu(\cc)$ (in this case automatically $J(f) = \pu(\cc)$) if and only if $f$ is a Lattès example.

 \section{Non-Archimedean dynamical Green function}

Dynamics of rational functions over non-Archimedean valued fields has been developing rapidly during the past 
20 years, greatly inspired by the analogy with holomorphic dynamics. Most of the results of the previous section have 
analogues in the non-Archimedean setting.  We will not dwell upon the details in these notes, and rather refer the interested 
reader to \cite{silverman, baker_rumely}. 
We shall content ourselves with recalling the vocabulary of valued fields and heights, and the 
construction of the dynamical Green function.

 \subsection*{Vocabulary of valued fields} 
 To fix notation and terminology let us first  recall a few standard notions and facts.
 \begin{defi} 
 An \emph{absolute value} $\abs{\cdot}$ on a field $K$ is a function $K \to \re^+$ such that 
 \begin{itemize}
 \item for every $\alpha \in K$, $\abs{\alpha}\geq 0$ and $\abs{\alpha }=0$ iff $\alpha = 0$;
 \item for all $\alpha$, $\beta$, $\abs{\alpha \beta}  = \abs{\alpha}\cdot \abs{\beta}$;
  \item for all $\alpha$, $\beta$, $\abs{\alpha +  \beta}  \leq  \abs{\alpha} + \abs{\beta}$.
  \end{itemize}
  If in addition $\abs{\cdot}$ satisfies the ultrametric triangle inequality 
 $\abs{\alpha +  \beta}  \leq  \max( \abs{\alpha} , \abs{\beta})$, then it is said non-Archimedean.
 \end{defi}
 
Besides the usual modulus on $\cc$, a basic example is the $p$-adic norm on $\Q$, defined by
 $\abs{\alpha}_p = p^{-v_p(\alpha)}$ where $v_p$ is the $p$-adic valuation:  $v_p(a/b) = \ell$ where 
 $a/b  = p^\ell (a'/b')$ and $p$ does not divide $a'$ nor $b'$.
 On any field we can define the trivial absolute value by $\abs{0} = 0$ and $\abs{x}=1$ if $x\neq 0$. 
 
   If $(K, \abs{\cdot})$ is
  a non-Archimedean valued field, we   define the
 {\em spherical metric}\index{Sherical metric}\index{Metric>spherical=(Spherical ---)} on $\mathbb{P}^1(K)$ by
$$\rho(p_1, p_2) = \frac{\abs{x_1y_2-x_2y_1}}{\max(\abs{x_1} , \abs{y_1})\max(\abs{x_2} , \abs{y_2})},$$
where
$ p_1  =[x_1:y_1] $ and $
 p_2  =[x_2:y_2]$. From the ultrametric inequality we infer that the $\rho$-diameter of   
$\mathbb{P}^1(K)$ equals 1. 

\begin{defi} Two absolute values $\abs{\cdot}_1$ and $\abs{\cdot}_2$ 
on $K$ are said equivalent if there exists a positive real number $r$ such that 
$\abs{\cdot}_1 = \abs{\cdot}_2^r$. 

An equivalence class of absolute values on $K$ is called a \emph{place}. The set of places of $K$ is denoted by 
$\mathcal{M}_K$. 
\end{defi}

\begin{thm}[Ostrowski]\index{Ostrowski}\index{Theorem>ofostrowski=of Ostrowski}
A set of representatives of the set of places of $\Q$ is given by 
\begin{itemize}
\item the trivial absolute value;
\item the usual Archimedean absolute value $\abs{\cdot}_\infty$;
\item the set of $p$-adic absolute values $\abs{\cdot}_p$. 
\end{itemize}
\end{thm}
 
 Decomposition into prime factors  yields the {\em product formula} 
$$\forall x\in \Q^*, \   \prod_{p \in \mathcal{P}\cup\set{\infty}}\abs{x}_p =1,$$
where $\mathcal{P}$ denotes the set of prime numbers.

For number fields, places can also be described. First one considers $\Q_p$ the completion of 
$\Q$ relative to the distance induced by $\abs{\cdot}_p$. The $p$-adic absolute value extends to  
an algebraic closure of $\Q_p$,   so
we may consider $\cc_p$ the completion of this algebraic closure. (For $p= \infty$ of 
course $\cc_p = \cc$.) 

A number field $K$ 
admits a number of embeddings $\sigma:K\hookrightarrow \cc_p$, which define by restriction a family of absolute values
$a\mapsto \abs{\sigma(a)}_p$ on $K$. Distinct embeddings may induce the same absolute value. For any non-trivial
absolute value 
$\abs{\cdot}_v$ on $K$ there exists $p_v\in \mathcal{P}\cup\set{\infty}$ 
such that $\abs{\cdot}_v\rest{\Q}$ is equivalent to 
$\abs{\cdot}_{p_v}$ and $\e_v$ distinct embeddings $\sigma:K \hookrightarrow \cc_{p_v}$ such that 
the restriction of $\abs{\cdot}_{\cc_{p_v}}$ to ${\sigma(K)}$  induces $ \abs{\cdot}_v$ up to equivalence. 
Thus for each place of $K$  we can define an  absolute value $\abs{\cdot}_v$ on $K$ that is 
induced by $\abs{\cdot}_{\cc_{p_v}}$ by embedding $K$ in $\cc_{p_v}$ and 
  the following product formula    holds
  \begin{equation}
  \label{eq:product formula}
  \forall x\in K^*, \ \prod_{v\in \mathcal{M}_K} \abs{x}_v^{\e_v} 
 = \prod_{p\in \mathcal{P}\cup\set{\infty}}\prod_{\sigma: K \hookrightarrow \cc_p}\abs{\sigma(x)}_p=1.
 \end{equation}
A similar formalism holds for certain non-algebraic extensions of $\Q$, giving rise to a general notion of 
{\em product formula field}.

 \subsection*{Dynamical Green function} Let $(K, \abs{\cdot})$ be a complete valued field. 
 If $f$ is a rational map on $\pp^1(K)$ we can define a dynamical Green function   as before, by working 
 with homogeneous lifts to $\aa^2(K)$.  For $(x,y) \in \aa^2(K)$ we put 
 $\norm{(x,y)} = \max(\abs{x} , \abs{y})$, and fix a lift $F:  \aa^2(K)\to \aa^2(K)$ of $f$. 
 
 \begin{thm}\label{thm:green NA}
For every $(x,y)\in \aa^2(K)$, the limit $$G_F(x,y)  := \lim_{n\to \infty} \unsur{d^n} \log \norm{F^n(x,y)}$$ exists.
 
  The function $G_F$ is continuous on $\aa^2(K)$ (relative to the norm topology) and satisfies 
  \begin{itemize}
  \item $G_F\circ F = dG_F$;
  \item $G_F(x,y)  = \log{\norm{(x,y)}}  + O(1)$ at infinity;
  \item $G_F(\lambda x,\lambda y) = \log\abs{\lambda} +G_F(x,y)$.
  \end{itemize}
 \end{thm}

Note that the function $G_F$ depends on the chosen lift, more precisely $G_{\alpha F}  = \unsur{d-1} \log\abs{\alpha}+ G_F$.
The proof of the theorem is similar to that of the Archimedean case. The key point is  that there exists positive numbers 
$c_1$ and $c_2$ such that  the inequality 
\begin{equation}\label{eq:norm lift}
c_1 \leq \frac{\norm{F(x,y)}}{\norm{(x,y)}^d} \leq c_2
\end{equation}
 holds. The upper bound in \eqref{eq:norm lift} is obvious, and the lower bound is a consequence of the Nullstellensatz.

 What is more delicate is to make sense of  the measure theoretic part, i.e. to give a meaning to $\Delta G$. For this one 
 has to introduce  the formalism of 
  Berkovich spaces, which is  well-suited to  measure theory (see \cite{silverman, baker_rumely} for details).  
  
  If $f$ is a polynomial, instead of lifting to $\aa^2$, we can directly consider the Green function on $K$ as in the complex case. More precisely: if $f$ is a polynomial of degree $d$, 
 the sequence of functions $d^{-n} \log^+ \abs{f^n}$ converges locally uniformly to a  function 
 $G_f$ on $K$    satisfying $G_f  \circ f = d G_f$.

 \section{Logarithmic height} 
 We refer to \cite{acl2, silverman} for details and references on the material in this section. 
 For $x\in \pp^1(\Q)$ we can write $x= [a:b]$ with $(a,b)\in \zz^2$ and $\mathrm{gcd}(a,b) = 1$, we define the {\em naive  height}\index{Naive=Na\"ive>height}\index{Height>naive=(Na\"ive ---)} $h(x)   = h_\mathrm{naive}(x)
   = \log\max(\abs{a}_\infty, \abs{b}_\infty)$. A fundamental (obvious) property is that for every $M\geq 0$
   $$\set{x \in \pp^1(\Q), h (x)\leq M}$$ is finite. 
   
 We now explain how to extend $h$ to $\pp^1(\overline \Q)$ using 
   the formalism of \emph{places}. The key observation (left as an exercise to the reader) is that   the 
   product formula implies that 
   $$   h(x) = \sum_{p\in \mathcal{P}\cup\set{\infty}} \log^+\abs{x}_p \ . $$  
      This formula can   be generalized to number fields, giving rise to a notion of {\em logarithmic height}\index{Logarithmitic height}\index{Height>logarithmic=(Logarithmic ---)} in that setting. 
   Let $K$ be a number field and $x = [x_0, x_1]\in \pp^1(K)$. Then with notation as above we define 
\begin{align*}
h(x)  &= \unsur{[K:\Q]} 
 \sum_{p\in \mathcal{P}\cup\set{\infty}}\sum_{\sigma: K \hookrightarrow \cc_p} 
 \log\norm{(\sigma(x_0) ,  \sigma(x_1)) }_p  \\
&=\unsur{[K:\Q]} 
 \sum_{v\in \mathcal{M}_K} \e_v  \log\norm{(\sigma(x_0) ,  \sigma(x_1)) }_v,
 \end{align*}
 where as before $\norm{(x_0,x_1)}_v = \max(\abs{ x_0}_v, \abs{ x_1}_v )$.Note  that 
  by the product formula, $h(x)$ does not depend on the lift $(x_0,x_1)\in \aa^2$. 
  By  construction, $h$ is invariant under the action of the Galois group $\mathrm{Gal}(K:\Q)$. 
  Furthermore, 
  the normalization in the definition of $h$ was chosen in such a way that  
  $h(x)$ does not depend on the choice of the number field containing $x$, therefore $h$ is a well-defined function on 
$\pp^1(\overline \Q)$ invariant under the absolute Galois group of $\Q$. 

 If $x\in K$ we can write $x=[x:1] $ and we recover 
 $$\log \norm{(x,1)}_p = \log\max(\abs{ x_0}_p, \abs{  x_1 }_p ) =  \log^+\abs{x}_p\;.$$  
 \begin{thm}[Northcott]\index{Northcott}\index{Theorem>ofnorthcott=of Northcott}
 For all $M>0$ and $D\leq 1$, the set 
 $$\set{x \in \pp^1(\overline \Q), h (x)\leq M \text{ and } [\Q(x):\Q]\leq D}$$ is finite. 
 \end{thm}
 
 \begin{cor}\label{cor:northcott}
For $x\in \overline \Q^*$,  $h(x)  = 0$ if and only if $x$ is a root of unity.
 \end{cor}
\begin{proof} Indeed, note that by construction $h(x^N) = Nh(x)$ and also $h(1) = 0$.  In particular if $x$ is a root of unity, then $h(x) = 0$. Conversely if $x\neq 0$ and $h(x) = 0$ then by Northcott's Theorem, $\set{x^N, \ N\geq 1}$ is finite so 
there exists $M<N$ such that $x^M=x^N$ and we are done. 
\end{proof}

 \subsection*{Action under rational maps} Let now $f:\pp^1(\overline \Q)\to\pp^1(\overline \Q) $ 
 be a rational map defined on $\overline \Q$,  of degree $d$. 
 
 \begin{prop}[Northcott] \label{prop:northcott}
 There exists a constant $C$ depending only on $f$ such that 
 $$\forall x\in \pp^1(\overline \Q), \ \abs{h(f(x)) - dh(x)}\leq C.$$
 \end{prop}
 
 In other words, $h$ is almost multiplicative. The proof follows from the decomposition of $h$ as a sum of 
 local contributions together with the Nullstellensatz as in  \eqref{eq:norm lift}
  (note that $h(f(x))  = h(x)$ for all but finitely many places).  
 
  We are now in position to define the {\em canonical height}\index{Canonical>height}\index{Height>canonical=(Canonical ---)} associated to $f$. 
 
 \begin{thm}[Call-Silverman]\label{thm:call silverman}
 There exists a unique function $h_f: \pp^1(\overline \Q) \to \re_+$ such that: 
 \begin{itemize}
 \item $h  - h_f$ is bounded;
 \item $h_f\circ f = d\, h_f$. 
 \end{itemize}
 \end{thm}
 
 Note that the naive height $h$ is the canonical height associated to $x\mapsto x^d$. 
 
 \begin{proof} This is similar  to the construction of the dynamical Green function. For every $n\geq 0$, let 
 $h_n = d^{-n} h\circ f^n$. By Proposition \ref{prop:northcott} we have that $\abs{h_{n+1} - h_n }\leq C/d^n$, 
 so the sequence $h_n$ converges pointwise. We define $h_f$ to be its limit, and the announced properties 
 are obvious. 
 \end{proof}
 
 \begin{prop} 
 The canonical height  $h_f$ enjoys the following properties:
 \begin{itemize}
 \item $h_f\geq 0$;
 \item for $x\in K^*$,  $h_f(x)  = 0$ if and only if $x$ is preperiodic, i.e. there exists $k<l$ such that $f^k(x) = f^l(x)$. 
 \item the set $\displaystyle \set{x \in \pp^1(\Q), h (x)\leq M \text{ and } [\Q(x):\Q]\leq D}$ is finite. 
 \end{itemize}
 \end{prop}
 
 \begin{proof}
 The first item follows from the positivity of $h$ and 
  the formula $h_f(x ) = \lim_n d^{-n} h(f^n(x))$. The second one is proved exactly as Corollary
 \ref{cor:northcott}, and the last one follows from the estimate $h_f - h = O(1)$. 
 \end{proof}

 It is often useful to know that the logarithmic height expresses as a sum of local contributions coming 
 from the   places of $K$. For the canonical height, the existence of such a decomposition 
 does not clearly follow from 
 the Call-Silverman definition. 
On the other hand the construction of the dynamical Green function from the previous paragraph implies that 
 a similar description holds for $h_f$. 
 
 \begin{prop}\label{prop:h_places}
 Let $f$ be a rational map on $\pu(\overline \Q)$ and $F$ be a homogeneous lift of $f$ on $\aa^2(\overline \Q)$. 
 Then for every $x\in \pu(\overline \Q)$, 
 we have that 
 $$h_f(x)  = \unsur{[K:\Q]} 
 \sum_{v\in \mathcal{M}_K} \e_v  G_{F, v}(x_0, x_1), \text{  where } x=[x_0:x_1]$$ where $K$ is any number field such that 
 $x\in \pu(K)$, $ G_{F, v}$ is the dynamical Green function associated to $(K, \abs{\cdot}_v)$ and $\e_v$ is as in \eqref{eq:product formula}.  
  \end{prop}
 
 \begin{proof}
 Let $\kappa(x_0, x_1)$ denote the right hand side of the formula. 
 The product formula together with the  identities  $G_F(\lambda x,\lambda y) = \log\abs{\lambda} +G_F(x,y)$ and 
 $G_{\alpha F} = G_F+\unsur{d-1}\log\abs{\alpha}$  imply that  $\kappa$ is independent 
 of the lifts $F$ of $f$ and $[x_0:x_1]$ of $x$. In particular $\kappa$ is a function of $x$ only.
 
 The invariance relations for the local Green functions imply that 
 $\kappa\circ f = d\; \kappa$. Since furthermore $G_{F, v}(x_0, x_1) = \log\norm{(x_0, x_1) }_v$ for all but finitely many places, 
 we infer from the second property of Theorem \ref{thm:green NA} that $\kappa-h = O(1)$. Then the result follows from 
 the uniqueness assertion in Theorem \ref{thm:call silverman}. 
 \end{proof}

 \section{Dynamical consequences of arithmetic equidistribution}\label{sec:dynamics_equidist}
 
\subsection*{Equidistribution of preperiodic points} 
It turns out that  the Call-Silverman canonical height satisfies the assumptions of Yuan's equidistribution theorem for points of small height (see \cite[Thm. 7.4]{acl}). From this for rational maps defined over number fields 
we can obtain  refined versions of previously known equidistribution theorems. 

\begin{thm}\label{thm:dynamics_equidist}
Let $f\in \overline \Q(X)$ be a rational map of degree $d\geq 2$ and $(x_n)$ be any infinite sequence of preperiodic 
points of $f$. Then in $\pu(\cc)$ we have that 
\begin{equation}\label{eq:yuan}
\frac{1}{[\Q(x_n):\Q]} \sum_{y\in G(x_n)} \delta_y \underset{n\to\infty}\longrightarrow \mu_f,
\end{equation}
where $G(x_n)$ denotes  the set of Galois conjugates of $x_n$. 
\end{thm}

 The reason of the appearance of $\mu_f$ in \eqref{eq:yuan} is to be found in the description of the canonical height 
 given in  Proposition \ref{prop:h_places}. 
When $G(x_n)$ is the set of all periodic points of a given period, this result is  a direct consequence of the last item of 
Theorem \ref{thm:equilibrium}. 
Indeed, by Theorem \ref{thm:fatou_shishikura}, $f$ admits at most finitely many non-repelling points.  
  Likewise, we may also obtain in the same way an arithmetic proof of the convergence of the measures $\mu_{n, z}$ of Theorem 
 \ref{thm:equilibrium}.

 \subsection*{Rigidity}  As an immediate consequence of Theorem \ref{thm:dynamics_equidist} we obtain the following rigidity statement. 
 
 \begin{cor}\label{cor:rigidity}
If $f$ and $g$ are rational maps  defined over  $\overline \Q$ of degree at least 2 are such that $f$ and $g$ have 
 infinitely many common preperiodic points, then $\mu_f  = \mu_g$ (so in particular $J(f) = J(g)$). 
\end{cor}
   
 Rational maps with the same equilibrium measure were classified  by Levin and Przytycki~\cite{levin_przytycki}. 
 The classification  is  a bit  difficult to state precisely, but the idea is  that if $f$ and $g$ have the same equilibrium measure
 and do not belong to a small list of well understood exceptional 
 examples (monomial maps, Chebychev polynomials, and Lattès maps)
 then  some iterates  $f^n$ and $g^m$ are  related by a certain correspondence.  
   If $f$ and $g$ are polynomials of the same degree, then the answer becomes simple: $\mu_f =\mu_g$  if and only 
  if there exists a linear transformation 
 $L$ in $\cc$ such that $g=L\circ f$ and $L(J) = J$, where $J = J(f)  = J(g)$. 
 A consequence of this classification is that for non-exceptional $f$ and $g$, 
  if $\mu_f= \mu_g$ then $f$ and $g$ have the same sets of preperiodic points 
(see \cite[Rmk 2]{levin_przytycki}). 
   
  Note also that  Yuan's equidistribution theorem   says more:   a meaning can be given 
  to the    convergence \eqref{eq:yuan} at   finite places, and indeed 
 Theorem \ref{thm:dynamics_equidist} holds at all places. It then follows from the description of the canonical height 
 given in Proposition \ref{prop:h_places} that if
  $f$ and $g$  are as in Corollary \ref{cor:rigidity}, then $h_f = h_g$, 
  in particular the sets of preperiodic points of  $f$ and $g$ coincide. 
    Denote by $\preper(f)$ the set of preperiodic points of $f$. 
   
  \begin{thm} 
  If $f$ and $g$ are rational maps   of degree at least 2, defined over $\cc$.  Consider the  following    conditions:
  \begin{enumerate} 
  \item[(1)] $\#\preper(f)\cap \preper(g) = \infty $;
  \item[(2)] $\preper(f) = \preper(g) $;
  \item[(3)] $ \mu_f = \mu_g$.
\end{enumerate}
Then    (1) is equivalent to (2) and both imply (3).  If $f$ and $g$ are not exceptional then the converse implication holds as well. 
  \end{thm}
 
 \begin{proof}
 This statement was already fully justified   when  $f$ and $g$ have  their coefficients in $\overline \Q$. Note also that 
the implication  $(3) \Rightarrow (2)$     proven in    \cite{levin_przytycki} and mentioned above 
  holds for arbitrary non-exceptional  
rational  maps. The equivalence between (1) and (2) for maps with complex coefficients was established by  
 Baker and DeMarco~\cite{baker_demarco1} (see also \cite[Thm. 1.3]{yuan_zhang}),  using some equidistribution results  over arbitrary valued fields as well as specialization arguments. Finally, the implication $(2)\Rightarrow (3)$ in this general setting was obtained by Yuan and Zhang~\cite[Thm. 1.5]{yuan_zhang}.
 \end{proof}

 \part{Parameter space questions}
 
 The general setting in this second part is the following: let $(f_\lambda)_{\la\in \La}$ be a holomorphic family of rational maps of 
 degree $d\geq 2$ on $\pu(\cc)$, that is for every $\la$, $f_\la(z)  = \frac{P_\la(z)}{Q_\la(z)}$ and $P_\la$ and $Q_\la$ depend holomorphically on $\la$ and have no common factors for every $\la$. Here  
   the parameter space $\Lambda$ is an arbitrary complex manifold (which may be the space of all rational maps of degree $d$). 
   From the dynamical point of view we can define a natural dichotomy $\La  = \stab\cup\bif$ of the parameter space 
   into an open {\em stability 
   locus}\index{Stability locus}\index{Locus>stability=(Stability ---)} $\stab$,\glossary{Stab: Stability locus} where the dynamics is in a sense locally constant, and its complement the {\em bifurcation locus}\index{Bifurcation>locus}\index{Locus>bifurcation=(Bifurcation ---)} $\bif$.\glossary{Bif: Bifurcation locus} Our purpose is to show how arithmetic ideas can give interesting information on these parameter spaces.

   \section{The quadratic family}\label{sec:quadratic}
   
   The most basic  example of this parameter dichotomy is given by the family of quadratic polynomials. Any degree 2 
   polynomial is affinely conjugate to a (unique) polynomial of the form $z\mapsto z^2+c$. so dynamically speaking the family of 
   quadratic polynomials {\em is} $\set{f_c: z\mapsto z^2+c , c\in \cc}$.  It was studied in great depth 
   since the beginning of the 1980's, starting with  the classic monograph by Douady and Hubbard \cite{douady_hubbard}. 
   
   Let $f_c(z) = z^2+c$ as above. A first observation is that if $\abs{z}>\sqrt{\abs{c}}+ 3$ then 
\begin{equation}\label{eq:escape}
\abs{z^2+c}\geq \abs{z}^2 - \abs{c}
    \geq 2 \abs{z}+1
    \end{equation} so $f_c^n(z) \to \infty$. Hence we deduce that 
      the      filled  Julia set $K_c   = K({f_c})$ is contained in $D(0, \sqrt{\abs{c}}+ 3)$. Recall that the Julia set of $f_c$ 
       is $J_c = \fr K _c$.  Note that the critical point is 0. 
       
       \subsection*{Connectivity of $J$} Fix $R$ large enough so that 
       $f\inv(D(0, R))\Subset D(0, R)$ ($R= (\abs{c}+1)^2+1$ is enough) then by the maximum principle, $f\inv(D(0, R))$ is a union of simply connected open sets (topological disks). 
       Since $f :  f\inv(D(0, R))\to D(0, R)$ is proper it is a branched covering  of degree 2
       so the topology of  
       $ f\inv(D(0, R))$ can be determined from the Riemann-Hurwitz formula. There are two possible cases:
     \begin{itemize}
     \item Either $f(0) \notin D(0, R)$. Then  $f\rest{ f\inv(D(0, R))}$ is a covering and $ f\inv(D(0, R)) = U_1\cup U_2$ is the union of 
     two  topological disks on which $f\rest{U_i}$ is a biholomorphism. Let    $g_i = (f\rest{U_i})\inv:D(0, R) \to U_1 $, which is 
     a contraction for the Poincaré metric in $D(0, R)$.  It follows that 
     $$K_c  = \bigcap_{n\geq 1} \bigcup_{(\e_i)  \in \set{1, 2}^n } g_{\e_n}\circ \cdots \circ   g_{\e_1} (D(0, R))$$  is a 
     Cantor set. 
   \item Or $f(0) \notin D(0, R)$. Then $ f\inv(D(0, R))$ is a topological disk and $f: f\inv(D(0, R))\to D(0, R)$ is a branched covering of degree 2. 
   \end{itemize}
 More generally, start with $D(0, R)$ and pull back under $f$ until $0\notin f^{-n}((D(0,R))$. Again there are two possibilities:
 \begin{itemize}
 \item Either this never happens. Then $K_c$ is a nested intersection of topological disks, so it is a connected compact set which may
 or may not have interior. 
 \item Or this happens for some $n\geq 1$. Then the above reasoning shows that $K_c = J_c$ is a Cantor set. 
   \end{itemize}
   
   Summarizing the above discussion, we get the following alternative:
    \begin{itemize}
     \item either $0$ escapes under iteration and $K_c = J_c$ is a Cantor set;
     \item or $0$ does not escape  under iteration and $K_c$ is connected  (and so is $J_c$). 
   \end{itemize}
   We   define the {\em Mandelbrot set}\index{Mandelbrot set} $M$ to be the set of $c\in \cc$ such that $K_c$ is connected. 
   From the previous alternative we see that the complement of $M$ is an open subset of the plane. 
Furthermore it is easily seen from   \eqref{eq:escape} that  0 escapes  when  $\abs{c}>4$,  
    so $M$ is a compact set in $\cc$. The Mandelbrot set is full (that is, its complement has bounded component) and  
  has non-empty interior: indeed if the critical point is attracted by a periodic cycle for $c=c_0$, then this behavior persists for $c$ close to $c_0$ and $K_c$ is connected in this case. 
   
   The following proposition is an easy consequence of the previous discussion.
   
   \begin{prop}
   $c$ belongs to  $\fr M$ if and only if for every open neighborhood $N\ni c$, $(c\mapsto f^n_c(0))_{n\geq 1}$ is not a normal family on $N$. 
   \end{prop}

\subsection*{Aside: active and passive critical points} 
We will generalize the previous observation to turn it into a quite versatile concept. Let $(f_\la)_{\la\in \La}$ be a 
holomorphic family of rational maps as before. In such a family a   critical point for $f_{\lo}$ cannot always be followed 
holomorphically due to ramification problems (think of $f_\la(z) = z^3+ \la z$), however this can always be arranged by passing to a branched 
 cover of $\La$ (e.g. by replacing $z^3+ \la z$ by $z^3+\mu^2z$). 
 We say that a holomorphically moving critical point $c(\la)$ is {\em marked}.\index{Marked critical point}\index{Critical point>marked=(Marked ---)}

\begin{defi}
Let $(f_\la, c(\la))_{\la\in \La}$ be a holomorphic family of rational maps with a marked critical point. We say
that $c$ is \emph{passive} on some open set $\om$ if the sequence of meromorphic mappings
 $(\la\mapsto f_\la^n(c(\la)))_{n\geq 0}$ is normal in $\om$. Likewise we say that $c$ is \emph{passive}\index{Passive critical point}\index{Critical point>passive=(Passive ---)} at $\lo$ if it is passive in some 
 neighborhood of $\lo$. Otherwise $c$ is said \emph{active}\index{Active critical point}\index{Critical point>active=(Active ---)} at $\lo$. 
 \end{defi}

This is an important concept in the study of the stability/bifurcation theory of rational maps, according to the principle 
``the dynamics is governed by that of critical points''. The terminology is due to McMullen.  
The next proposition is a kind of 
parameter analogue of the density of periodic points in the Julia set. 

\begin{prop} \label{prop:levin}
Let $(f_\la, c(\la))_{\la\in \La}$ be a holomorphic family of rational maps of degree $d\geq 2$
 with a marked critical point. If $c$ is active at $\lo$ there exists an infinite sequence 
  $(\la_n)$ converging to $\lo$ such that 
 $c(\la_n)$ is preperiodic for $f_{\la_n}$. More precisely we can   arrange so that 
 \begin{itemize}
 \item either  $c(\la_n)$ falls under iteration  onto a repelling periodic point;
 \item or $c(\la_n)$ is periodic.
 \end{itemize}
\end{prop}
 
\begin{proof}
Fix a repelling cycle of length at least 3 for $f_{\lo}$. By the implicit function theorem, 
this cycle persists and stays repelling in some   neighborhood of $\lo$. More precisely if we fix 3 distinct points 
$\alpha_i(\lo)$, $i= 1,2,3$  in this cycle, there exists an open neighborhood $N(\lo)$ and 
holomorphic maps $\alpha_i: N(\lo)\to \pu$  such that $\alpha_i(\la)$ are holomorphic continuations of 
$\alpha_i(\lo)$ as periodic points. Since the family $(f_\la^n(c(\la)))_{n\geq 1}$ is not normal in 
$N(\lo)$  by Montel's theorem, there exists $\la_1\in N(\lo)$,   $n\geq 1$ and $i\in \set{1,2,3}$ such that 
$f_{\la_1}^n(c(\la_1)) = \alpha_i(\la_1)$.  Thus the first assertion is proved, and the other one is similar
 (see \cite{levin} or  \cite{preper} for details).
\end{proof}

The previous result may be interpreted by saying that 
 an active critical point is a source of bifurcations.  Indeed, given any holomorphic family $(f_\la)_{\la\in \La}$
 of rational maps,   
 $\La$  can be written as the union of an open  {\em stability locus}\index{Stability locus} $\stab$ where  the dynamics is (essentially) 
 locally constant on $\pu$ and its complement, the {\em bifurcation locus}\index{Bifurcation>locus} $\bif$ where the dynamics
changes drastically. 
When all critical points are marked, the bifurcation locus is exactly 
 the union of the activity loci
of the critical points.
It is a fact that in any holomorphic family, the bifurcation locus is a fractal set 
with rich topological structure.

\begin{thm}[Shishikura \cite{shishikura}, Tan Lei \cite{tanlei}, McMullen \cite{mcmullen_universal}]\label{thm:shishikura}
If $(f_\la)_{\la\in\La}$ is any holomorphic family of rational maps with non-empty bifurcation locus,   
the Hausdorff dimension of $\bif$ is equal to that of $\La$. 
\end{thm}

Recall that the Hausdorff dimension of a metric space is a non-negative number which somehow measures the 
scaling behavior of the 
metric. For a manifold, it coincides with its topological dimension, but for a fractal set it is typically not an integer. 
 In the quadratic family, the bifurcation locus is the boundary of the Mandelbrot set, and 
  Shishikura  \cite{shishikura} proved that $\mathrm{HDim}(\fr M) = 2$. Therefore $\fr M$ is a nowhere dense subset 
which in a sense locally fills the plane.   It was then shown in 
  \cite{tanlei, mcmullen_universal} that this topological complexity can be transferred to  any holomorphic family, 
  resulting in  the above theorem.

   \subsection*{Post-critically finite parameters in the quadratic family} 
   
   A rational map is said {\em post-critically finite}\index{Post-critically finite rational map}\index{Rational map>postcriticallyfinite=(Post-critically finte ---)} if its critical set has a finite orbit, that is, every critical point is periodic 
   or preperiodic. The post-critically finite parameters in the quadratic family are the solutions of the countable family of 
   polynomial equations $f_c^k(0) = f_c^l(0)$, with $k>l\geq 0$, and form a countable set of ``special points'' in parameter
   space. 
   
   As a consequence of  Proposition \ref{prop:levin} in the quadratic family we get 
   
   \begin{cor} $\fr M$ is contained in the closure of the set of post-critically finite parameters. 
    \end{cor}
   
Post-critically finite quadratic polynomials can be of two different types:
\begin{itemize}
\item either  the critical point 0 is periodic. in this case 
   $c$ lies in the interior of $M$. Indeed
    the attracting periodic 
   orbit persists in some  neighborhood of $c$, thus persistently attracts the critical point and $K_c$ remains connected;
   \item or $0$ is strictly preperiodic. Then  it can be shown  that 
 it must fall on a repelling cycle, so it is active and $c\in \fr M$.   
\end{itemize}

 The previous corollary can be strengthened to an equidistribution statement. 
 
 \begin{thm} \label{thm:equidist mandelbrot}
 Post-critically finite parameters are asymptotically equidistributed.
 \end{thm}

There are several ways of formalizing this. For a pair of integers $k>l\geq 0$, denote by $\percrit(k,l)$ 
the (0-dimensional) variety 
defined by the polynomial equation $f_c^k(0) = f_c^l(0)$, and by $[\percrit(k,l)]$ the sum of point masses at the 
corresponding points, counting multiplicities. 

Then the precise statement of the  theorem is that there exists a probability measure $\mu_M$ on the Mandelbrot 
set such that if $0\leq k(n) < n$ is an arbitrary sequence, then 
$2^{-n} [\percrit(n,k(n))] \to \mu_M$ as $n\to \infty$. Originally proved by 
Levin \cite{levin_equidist} (see also \cite{mcm_equidist} for $k(n)=0$), this 
result was   generalized by several authors
 (see e.g. \cite{preper, baker_hsia}). Quantitative estimates on the speed of convergence 
are also available \cite{favre_rivera, gauthier_vigny}. 

We   present an   approach to this result   based on arithmetic equidistribution, 
along the lines of \cite{baker_hsia, favre_rivera}. 
 
Recall the dynamical Green function 
$$G_{f_c}(z )  = \lim_{n\to\infty} \frac{1}{2^n} \log^+\abs{f_c^n(z)},$$
which is a non-negative continuous and plurisubharmonic function of $(c,z)\in \cc^2$. Put $G_M(c) = 
G_{f_c}(c) = 2 G_{f_c}(0)$. This function is easily shown to have  the following properties:
\begin{itemize}
\item $G_M$ is non-negative, continuous and subharmonic on $\cc$; 
\item $G_M(c) = \log\abs{c}+ o(1)$ when $c\to \infty$;
\item $\set{G_M =0}$ is the Mandelbrot set;
\item $G_M$ is harmonic on $\set{G_M >0} =M^\complement$.
\end{itemize}
Therefore $G_M$  is the (potential-theoretic) Green function of the Mandelbrot set and $\mu_M:=\Delta G_M$ is 
the harmonic measure of $M$. 
 
 To apply arithmetic equidistribution theory, we need to understand what happens at the non-Archimedean places. 
 For a prime number $p$, let 
 $$M_p = \set{c\in \cc_p, \ (f^n_c(0))_{n\geq 0} \text{ is bounded in } \cc_p}.$$
      
  \begin{prop}~
  
  \begin{enumerate}[{(i)}]
  \item For every $p\in \mathcal{P}$, $M_p$ is the closed unit ball of $\cc_p$.
  \item For every $p\in \mathcal{P}$, for every $c\in \cc_p$, 
  $$G_{M_p}(c) = G_{f_c}(c)   = \lim_{n\to\infty} 2^{-n} \log^+\abs{f^n_c(c)}_p = \log^+\abs{c}_p.$$
  \item The associated height function $h_{\mathbb{M}}$ defined for $c\in \overline \Q$ by 
  $$h_{\mathbb{M}}(c)   = \unsur{[\Q(c):\Q]} \sum_{p\in \mathcal{P}\cup\set{\infty}} \sum_{\sigma: \Q(c) \hookrightarrow \cc_p}
  G_{M_p}(\sigma(c))$$ satisfies 
  $$\set{c\in \overline \Q, \ h_{\mathbb{M}}(c) = 0} = \bigcup_{0\leq k <n} \percrit(k,n).$$
  \end{enumerate}
  \end{prop}
         
         The collection $\mathbb{M}$ of the sets  $M_p$ for $p\in \mathcal{P}\cup\set{\infty}$ is called the {\em adelic Mandelbrot set}\index{Adelic>Mandelbrot set}\index{Mandelbrot set>adelic=(Adelic ---)}, and $h_\mathbb{M}$ will be referred to as the {\em parameter height function} associated to the critical point 0.

    \begin{proof}
    Using the ultrametric property, we see that $\abs{c}_p \leq 1$ implies that 
    $\abs{f_c(c)}_p = \abs{c^2+c}_p\leq 1$. Conversely  $\abs{c}_p > 1$ implies   
    $\abs{c^2+c}_p = \abs{c}_p^2$ hence $f_c^n(c)\to \infty$. This proves {\em (i)} and   {\em (ii)}. 
    
 For the last assertion, we observe that    the canonical height of $f_c$ is  given by 
 $$h_{f_c}(z)   = \unsur{[\Q(z):\Q]} \sum_{p\in \mathcal{P}\cup\set{\infty}} \sum_{\sigma: \Q(z) \hookrightarrow \cc_p}
  G_{f_c, p}(\sigma(z)).$$ Indeed it is clear from this formula that $h_{f_c}(z) = 2 h_{f_c}(z)$ and $h_{f_c}- h_{\rm naive} =
   O(1)$ so the result follows from the Call-Silverman Theorem \ref{thm:call silverman}. Recall that 
   $h_{f_c} (z) = 0$ if and only if $z$ is preperiodic. Assertion {\em (iii)} follows from these 
   properties by simply plugging $c$ into the formulas. 
    \end{proof}

 An {\em adelic set}\index{Adelic>set} (this terminology is due to Rumely)
 $\mathbb{E} = \set{E_p, p\in \mathcal{P}\cup\infty}$ is a collection of sets $E_p\subset \cc_p$ such that 
 \begin{itemize}
 \item $E_\infty$ is a full compact set in $\cc$;
 \item for every $p\in \mathcal{P}$, $E_p$ is closed and bounded in $\cc_p$, and $E_p$ is the closed unit ball for 
 all but finitely many $p$;
 \item for every $p\in \mathcal{P}\cup\infty$, $E_p$ admits a Green function $g_p$ 
 that is continuous on $\cc_p$ and satisfies $E_p = \set{g_p = 0}$, 
 $g_p (z) = \log^+\abs{z}_p - c_p + o(1)$ when $z\to \infty$, and $g_p$ is ``harmonic''\footnote{We do not define precisely what ``harmonic'' means in the $p$-adic context: roughly speaking it means that locally
 $g = \log\abs{h}$ for some non-vanishing analytic function (see \cite[Chap. 7]{baker_rumely} for details on this notion). Note that for the adelic Mandelbrot set  
 we have $g_p  = \log^+\abs{\cdot}_p$ at non-Archimedean places.} outside $E_p$.
  \end{itemize}
The \emph{capacity} of an adelic  $\mathbb{E}$ is 
defined to be  $\gamma(\mathbb{E}) =\prod_{p\in \mathcal{P}\cup\infty} e^{c_p}$. We will typically assume that 
$\gamma(\mathbb{E}) =1$.
 Under these assumptions one defines a height function from the local Green functions  $g_p$
exactly as before  $$h_\mathbb{E} (z) =
\unsur{[\Q(z):\Q]} \sum_{p\in \mathcal{P}\cup\set{\infty}} \sum_{\sigma: \Q(z) \hookrightarrow \cc_p}
  g_p(\sigma(z)),$$
 and we have the following  equidistribution theorem (see \cite{autissier, acl} for details and references).

 \begin{thm}[Bilu, Rumely]
Let $\mathbb{E}$ be an adelic set
  such that $\gamma(\mathbb{E}) =1$, and $(x_n)\in \cc^\nn$
   be a sequence of points with disjoint Galois orbits $X_n$ and such that
 $h(x_n)$ tends to 0. Then the sequence of equidistributed probability measures on the sets $X_n$ converges weakly in 
 $\cc$ to the 
 potential-theoretic equilibrium measure of $E_\infty$. 
 \end{thm}
 
 This convergence also holds at finite places, provided one is able to make sense of measure theory in this context. Theorem \ref{thm:equidist mandelbrot} follows immediately.

   \section{Higher degree polynomials and equidistribution}\label{sec:equidist higher}
   
   In this and the next section we investigate the asymptotic distribution of special points in  spaces of higher degree polynomials.  
   The situation is far less understood for spaces of rational functions. Polynomials of the form
    $P(z)  = \sum_{k=0}^d a_k z^k$
   can be obviously parameterized by 
   their coefficients so that the space $\poly_d$ of polynomials of degree $d$ is $\cc^*\times \cc^d$. By an affine conjugacy we can 
   arrange   that $a_d= 1$ and $a_{d-1} = 0$ (monic and centered polynomials).
   
    Assume now that $P$ and $Q$ are 
   monic,
   centered and conjugate by the affine transformation $z\mapsto az+b$, that is, $P(z) =  a\inv Q(az+b) - ba\inv$. Since 
  $P$ and $Q$ are monic we infer that $a^{d-1} = 1$ and from the centering we get $b=0$. It follows that the space $\mpoly_d$
  of polynomials of degree $d$
  modulo affine conjugacy is naturally isomorphic to 
  $\cc^{d-1} /  \langle \zeta \rangle$, where $\zeta  = e^{2i\pi/(d-1)}$. In practice it is easier to 
  work on its $(d-1)$-covering  by $\cc^{d-1}$.  
  
   \subsection*{Special points}
The special points\index{Special>points} in $\mpoly_d$ are the post-critically finite maps. They are dynamically natural since classical results of 
Thurston, Douady-Hubbard and others (see e.g. \cite{douady_hubbard_thurston,hubbard})
 show that  the geography of $\mpoly_d$  is somehow 
organized around them. 

 \begin{prop}\label{prop:PCF} \ 
 
 \begin{enumerate}[(i)]
 \item The set $\mathrm{PCF}$ of   \pcf polynomials is countable and Zariski-dense 
 in $\mpoly_d$.
 \item $\mathrm{PCF}$  is  
 relatively compact in the usual topology.    
 \end{enumerate}
  \end{prop}
 
 \begin{proof} For $d=2$ this follows from the properties of the Mandelbrot set so it is enough to deal with 
  $d\geq 3$. 
   To prove the result it is useful to work in the   space $\mpoly_c^{\mathrm{cm}}$
   of critically marked polynomials modulo affine conjugacy, which is a branched covering of $\mpoly_d$. 
   Again this space is singular so we   work on the following parameterization. For 
   $(c,a) = (c_1, \ldots , c_{d-2}, a)\in \cc^{d-1}$ we consider the polynomial $P_{c,a}$ 
   with critical points at $( c_0= 0,  ,c_1, \ldots , c_{d-2})$ and such 
   that $P(0)  = a^d$, that is $P_{c,a}$ is the primitive of $z\prod_{i=1}^{d-2}(z-c_i)$ such that $P(0)  = a^d$. This defines 
   a map $\cc^{d-1}\to \mpoly_c^{\mathrm{cm}}$ which is a branched cover of degree $d(d-1)$. 
   
  Let $G_{P_{c,a}}$ be the dynamical Green function of $P_{c,a}$, and define 
  $$G(c,a ) = \max \set{ G_{P_{c,a}}(c_i), \ i = 0 , \ldots , d-2}.$$ 
  The asymptotic behavior of $G$ is well understood. 
  
  \begin{thm}[Branner-Hubbard \cite{branner_hubbard}] \label{thm:branner hubbard}
  As $(c,a)\to \infty$  in $\cc^{d-1}$ we have  that 
  $$G(c, a)  = \log^+ \max \set{\abs{a},\abs{ c_i},   i = 1 , \ldots , d-2} + O(1).$$
   \end{thm}
   
 The seemingly curious normalization  $P(0) = a^d$ was motivated by this 
 neat expansion. See \cite{preper} for a proof using these 
 coordinates, based on explicit asymptotic expansions of the $P_{(c, a)}(c_i)$. 
   In particular $G(c, a) \to\infty$ as 
   $(c,a)\to \infty$ and it follows that the connectedness locus 
   $$\mathcal{C}:= \set{(c,a), \ K(P_{c,a}) \text{ is connected}} = \set{(c,a),\ G(c, a)  =0}$$ is
   compact. Since \pcf parameters belong to $\mathcal C$, this proves the second assertion of   Proposition \ref{prop:PCF}. 
   For the first one, note 
   that the set of \pcf parameters is defined by countably many algebraic equations in $\cc^d$, and since each component is
   bounded in $\cc^{d-1}$ it must be a point. 
   
  The Zariski density of \pcf polynomials is a direct consequence of the equidistribution 
  results of  \cite{bassanelli_berteloot, preper}, based on pluripotential theory. 
Here we present a simpler argument   due to 
  Baker and  DeMarco \cite{baker_demarco}.   This requires the following result, whose proof will be skipped. 
   
   \begin{thm}[McMullen \cite{mcm_algorithms}, Dujardin-Favre \cite{preper}] \label{thm:passive curve} 
  
   Let $(f_\la, c(\la))_{\la\in \La}$ be a holomorphic family of rational maps of degree $d$ with a marked critical point , parameterized
  by a quasiprojective variety $\La$. If $c$ is passive along $\La$, then:
  \begin{itemize}
  \item either the family is isotrivial, that is the $f_\la$ are conjugate by Möbius transformations
  \item or $c$ is persistently preperiodic, that is there exists $m<n$ such that 
  $f^m_\la(c(\la))\equiv f^n_\la(c(\la))$. 
  \end{itemize}
  \end{thm}
 
Let now  $S$ be any proper
algebraic subvariety of $\mpoly_d^{\mathrm{cm}} = : X$, we want to show that there exists a \pcf parameter in $X\setminus S$. Consider the first marked critical point $c_0$ on $\La = X\setminus S$. Since the family 
$(P_{c,a})$ is not isotrivial on $\La$, by the previous theorem, either it is persistently preperiodic or it must be active somewhere and by perturbation we can make it preperiodic 
by Proposition \ref{prop:levin}. In any case we can find $\la_0\in \La$ and $m_0<n_0$ such that 
   $f^{m_0}_\lo(c_0(\lo))\equiv f^{m_0}_\lo(c_0(\lo))$. Now define $\La_1$ to be the subvariety of codimension $\leq 1$
   of $X\setminus S$ where this 
   equation is satisfied. It is quasiprojective and $c_0$ is persistently preperiodic on it. We now consider the behaviour 
   of $c_1$ on $\La_1$ and continue inductively to get a nested sequence 
  of quasiprojective varieties $\La_k\subset X$ on which $c_0, \ldots , c_k$ are persistently preperiodic. 
 Since the dimension drops by at most 1 at each step and $\dim(X) = d-1$, we can continue until $k=d-2$ and 
 we finally find the desired parameter. 
 \end{proof}
  
  \subsection*{Equidistribution of special points} 
 Pluripotential theory (see e.g. \cite[Chap. III]{agbook}) allows 
  to give a meaning to the   exterior product $\lrpar{\frac{i}{\pi}\fr\overline \fr G} ^{\wedge (d-1)}$. 
  This  defines a probability measure with compact support in $\mpoly_d^{\mathrm{cm}} $ which 
      will be referred to as the {\em bifurcation measure}\index{Bifurcation>measure}\index{Measure>bifurcation=(Bifurcation ---)}, 
   denoted by  $\mu_{\rm bif}$.
  
  The next theorem asserts  that   \pcf parameters are asymptotically equidistributed.
  For $0\leq i\leq d-2$  and $m<n$ define the subvariety $\per_{c_i} (m,n)$ to be the closure 
  of the set of parameters at which  $f^k(c_i)$ is periodic exactly for $k\geq m$, with exact period $n-m$. 
  
  \begin{thm}[Favre-Gauthier \cite{favre_gauthier}]  \label{thm:favre gauthier} Consider a  
  $(2d-2)$-tuple of sequences of 
 integers   $((n_{k,0}, m_{k,0}), \ldots , (n_{k,d-2}, m_{k,d-2}))_{k\geq 0}$  such that: 
  \begin{itemize}
  \item either the $m_{k,i}$ are equal to 0 and for fixed $k$ the $n_{k,i}$ are distinct and $\min_i (n_{k,i} )$ tends to $
   \infty$ with $k$;
  \item or for every $(k,i)$, $n_{k,i} > m_{k, i} >0$ and    $\min_i (n_{k,i}- m_{k,i}) \to \infty$ when $k\to \infty$. 
  \end{itemize}
  Then, letting $Z_k = \per_{c_0}(m_{k,0}, n_{k,0}) \cap \cdots \cap  \per_{c_{d-2}}(m_{k,d-2}, n_{k,d-2})$, the sequence of probability measures uniformly distributed on $Z_k$   converges to the  bifurcation measure 
 as $k\to \infty$. 
   \end{thm}
  
 This   puts forward the bifurcation measure as the natural analogue in higher degree of the harmonic measure of 
 the Mandelbrot set. It was first defined and studied in \cite{bassanelli_berteloot, preper}. 
 It follows from its pluripotential-theoretic construction that $\mu_{\rm bif}$ carries  no mass on analytic sets. 
 In particular this gives another argument for the Zariski density of  \pcf parameters . 
  
  The result is a consequence of arithmetic equidistribution. 
  Using the function $G(c,a)$ and its adelic analogues, we can define as before a  parameter height function
  on $\mpoly_d^{\mathrm{cm}}( \overline \Q)$
  satisfying $h(c,a)=0$ iff 
  $P_{c,a}$ is critically finite. Then  Yuan's equidistribution theorem for points of small height
  (see \cite{acl}) applies in this situation --this requires 
  some non-trivial work on  understanding the properties of $G$ at infinity. Specialized 
  to our setting it takes the following form.
    
  \begin{thm}[Yuan]\label{thm:yuan}
   Let $Z_k\subset \mpoly_d^{\mathrm{cm}}( \overline \Q)$ be a sequence of Galois invariant subsets such that:
   \begin{enumerate}[(i)]
   \item $\displaystyle h(Z_k) = \unsur{\# Z_k}\sum_{x\in Z_k} h(x) \underset{k\to \infty}\longrightarrow 0$;
   \item For every algebraic hypersurface $H$ over $\Q$, 
   $H\cap Z_k$ is empty for large enough $k$.
   \end{enumerate}
   Then $   \unsur{\# Z_k}\sum_{x\in Z_k}  \delta_x $ converges to the  bifurcation measure $\mu_{\rm bif}$ as $k\to \infty$. 
   \end{thm}
   
    Actually in the application to \pcf maps 
   the genericity assumption {\em (ii)} is not   satisfied. Indeed we shall see in the next section that there are ``special subvarieties'' 
   containing infinitely many \pcf parameters.  Fortunately   the following variant is true. 
    
 \begin{cor}\label{cor:yuan}
 In Theorem \ref{thm:yuan}, if (ii) is replaced by the weaker condition: 
   \begin{itemize} 
   \item[{\it (ii')}] For every algebraic hypersurface $H$ over $\Q$, $\displaystyle \lim_{k\to\infty} \frac{\# H\cap Z_k}{\#Z_k} = 0$,
 \end{itemize}
 then the same conclusion holds. 
 \end{cor}
 
 \begin{proof}[Proof of the corollary]
 This is based on a diagonal extraction argument. First, enumerate all hypersurfaces   defined over $\Q$ to form 
 a sequence  $(H_q)_{q\geq 0}$. Fix $\e>0$. 
 For $q=0$ we have  that $\frac{\# H_0\cap Z_k}{\#Z_k} \to 0$ as $k\to \infty$ so if for $k\geq k_0$ we remove from 
 $Z_k$ the (Galois invariant) set of points belonging to $H_0$ to  get a subset   $Z_k^{(0)}$ such that 
 $$Z_k^{(0)}\cap H_0 = \emptyset \text{ and } \frac{\# Z_k^{(0)}}{\#Z_k} \geq 1 - \frac{\e}{4}.$$ For $k\leq k_0$ we put 
 $Z_k^{(0)} = Z_k$. 
 
 Now for $H_1$ we do the same. There exists $k_1>k_0$ and for $k\geq k_1$ a Galois invariant subset 
 $Z_k^{(1)}$ extracted from $Z_k^{(0)}$ such that 
 $$Z_k^{(1)}\cap H_1 = \emptyset \text{ and } \frac{\# Z_k^{(1)}}{\#Z_k^{(0)}} \geq 1 - \frac{\e}{8}.$$ 
 For $k< k_0$ we set $Z_k^{(1)} = Z_k$ and for  $k_0\leq k < k_1$ we set $Z_k^{(1)} = Z_k^{(0)}$
 Note that  for $k\geq k_0$, $Z_k^{(1)}\cap H_0 = \emptyset$. 
 
 Continuing inductively this procedure, for every $q\geq 0$ we get a sequence of subsets $(Z_k^{(q)})_{k\geq 0}$ 
 with 
 $Z_k^{(q)} \subset Z_k$   such that 
 $$\# Z_k^{(q)} \geq \prod_{j=0}^q \lrpar{1 - \frac{\e}{2^{j+2}}} \# Z_k \geq (1-\e) \# Z_k$$ and $ Z_k$ is disjoint from $H_0,\ldots 
 , H_q$ for $k \geq k_q$.
 Finally we define $Z_k^{(\infty)}   = \bigcap_q Z_k^{(q)}$, which satisfies that for every $q$ and every $k\geq k_q$, 
 $Z_k^{(\infty)} $ is disjoint from $H_0,\ldots 
 , H_q$ and for every $k \geq 0$, $\#  Z_k^{(\infty)}  \geq (1-\e) \# Z_k$.  Therefore $Z_k^{(\infty)} $
satisfies the assumptions of Yuan's Theorem so  
$$ \mu_k^{(\infty)}:=   \unsur{\# Z_k^{(\infty)}}\sum_{x\in Z_k^{(\infty)}}  \delta_x 
\underset{k\to\infty}\longrightarrow \lrpar{\frac{i}{\pi}\fr\overline \fr G} ^{\wedge (d-1)}.$$ Finally if we let $\mu_k$ be the uniform measure on $Z_k$, we have that for every continuous function 
$\varphi$ with compact support, 
$$\abs{\mu_k^{(\infty)} (\varphi)  - \mu_k  (\varphi)  }\leq 2\e\norm{\varphi}_{L^\infty}$$ and we conclude that
$\mu_k$ converges to the bifurcation measure as well. This finishes the proof. 
 \end{proof}
 
 \begin{proof}[Proof of Theorem \ref{thm:favre gauthier}]
 We treat   the first set of assumptions $m_{k,i} =0$, and denote 
 put $\per_{c_i}(n_{k,i})= \per_{c_i}(0,n_{k,i})$. It is convenient to assume that the $n_{k, i}$ are prime numbers, which
 simplifies the issues about prime periods.  We have  to check that the hypotheses Corollary \ref{cor:yuan}
 are satisfied. First $X_k$ is defined by $(d-1)$ equations over $\Q$ so it is certainly Galois invariant, and it is a set of \pcf parameters so 
 its parameter height vanishes.  So the point is to check condition {\em (ii')}. For every $0\leq i \leq d-2$
 the variety $\per_{c_i}(n_{k,i})$ is 
 defined by the equation $P^{n_{k,i}}_{c,a} (c_i) - c_i = 0$ which is of degree $d^{n_{k,i}}$. 
 Recall from Proposition \ref{prop:PCF} that $Z_k$ is of dimension 0 and relatively compact in 
 $\mpoly_d^{\mathrm{cm}}$. Furthermore the analysis leading to Theorem 
 \ref{thm:branner hubbard} shows that these hypersurfaces do not intersect at infinity 
 so by Bézout's theorem the cardinality of $Z_k$ equals $d^{n_{k,0}+ \cdots + n_{k,d-1}}$ {\em counting multiplicities}. 
 Actually 
 multiplicities do not account, due to the following deep 
 result, which is based on dynamical and Teichmüller-theoretic techniques.
 
 \begin{thm}[Buff-Epstein \cite{buff_epstein}] \label{thm:buff epstein}
 For every $x\in Z_k$, the varieties 
  $\per_{c_0}(n_{k,0})$, ..., $\per_{c_{d-2}}( n_{k,d-2})$
  are smooth and transverse at $x$.  
 \end{thm}

Finally, for every hypersurface $H$ we need to bound $\#H \cap Z_k$.  
Let $x\in H \cap Z_k$ and assume that $x$ is a regular point
of $H$. A first possibility is that 
 locally (and thus also globally) $H\equiv \per_{c_i}(n_{k,i})$ for some $i$. Since $n_{k,i}\to \infty$  this situation 
can happen only for finitely many $k$ so considering large enough $k$ we may assume that $H$  is distinct from the 
$\per_{c_i}(n_{k,i})$. By the transversality Theorem \ref{thm:buff epstein}, $H$ must be transverse to $(d-2)$ of the 
$\per_{c_i}(n_{k,i})$ at $x$ (this is the incomplete basis theorem!). So we can bound $\#H \cap Z_k$ by applying Bézout's theorem 
to all possible intersections of $H$ with $(d-2)$ of the 
$\per_{c_i}(n_{k,i})$, that is 
\begin{align*}
\#\mathrm{Reg} &(H) \cap Z_k    \\&\leq  \sum_{i=0}^{d-2} \# H \cap  \per_{c_0}(n_{k,0})\cap \cdots\cap \widehat{\per_{c_i}(n_{k,i})} \cap \cdots\cap
\per_{c_{d-2}}(n_{k,d-2}) \\
&\leq  \sum_{i=0}^{d-2} \deg(H) \cdot d^{\lrpar{\sum_{j=0}^{d-2} n_{k,j} }-n_{k,i}}   
= o\lrpar{d^{{\sum_{j=0}^{d-2} n_{k,j} }}}  = o\lrpar{\# Z_k}.
\end{align*}
To deal with the singular part of $H$, we write $\mathrm{Sing}(H) = \mathrm{Reg} (\mathrm{Sing}(H)) \cup\mathrm{Sing}(\mathrm{Sing}(H))$, and using the above argument in codimension 2, by we get a similar  estimate for 
$\#\mathrm{Reg} (\mathrm{Sing}(H)) \cap Z_k$. Repeating inductively this idea we finally 
conclude that $\#H \cap Z_k  =o\lrpar{\# Z_k}$, as required. 
 \end{proof} 
 
   \section{Special subvarieties}\label{sec:special}
   
      \subsection*{Prologue} 
    There are a number of situations in algebraic geometry where the following   happens:  an algebraic variety $X$ is given 
containing countably many ``special subvarieties''\index{Special>subvarieties} (possibly of dimension 0). Assume that a subvariety $Y$ admits a Zariski-dense subset of special points, then must it be special, too? Two famous instances of this problem are:
\begin{itemize}
\item torsion points on Abelian varieties (and the Manin-Mumford conjecture);
\item CM points on Shimura varieties (and the André-Oort conjecture).
\end{itemize}

The Manin-Mumford conjecture admits a dynamical analogue which will not be discussed in these notes. We will consider an analogue of the André-Oort conjecture which has been put forward by Baker and DeMarco \cite{baker_demarco}. Without entering into the details of what 
``André-Oort'' refers exactly to let us just mention one positive result which motivates the general conjecture. 
Let us identify the space of pairs of elliptic curves 
with $\cd$ via the $j$-invariant. 

\begin{thm}[André \cite{andre}]
Let $Y\subset \cd$ be an algebraic curve containing infinitely many points both coordinates of which are ``singular moduli'', that is, $j$-invariants of CM elliptic curves. Then $Y$ is special in the sense that $Y$ is either  a vertical or a horizontal line, or a modular curve $X_0(N)$. 
\end{thm}  
 
 Recall that $X_0(N)$ is the irreducible algebraic curve in $\cd$ uniquely defined by the property that 
 $(E, E')\in X_0(N)$  if there exists a cyclic  isogeny $E\to E'$ of degree $N$ (this property is actually symmetric in $E$ and $E'$). Likewise, it is characterized by the property that   for every $\tau \in \mathbb{H}$, 
 $(j(N\tau), j(\tau))\in X_0(N)$.  
 
   \subsection*{Classification of special curves} In view of the above considerations it is natural to attempt to classify special 
   subvarieties, that is subvarieties $\La$ 
   of $\mpoly_d$ (or $\mathcal{M}_d$) with a Zariski dense subset of \pcf parameters. Examples are easy to find: assume that the  $k$   critical points 
   $c_0, \ldots , c_{k-1}$ 
   critical points are ``dynamically related'' on a subvariety $\La$ of codimension $k-1$. This happens
     for instance if they satisfy a relation of the form
    $P^{k_i}(c_i) = P^{l_i}(c_0)$ for some  integers $k_i, l_i$ (say $\La$ is the subvariety cut out by $k-1$ such equations, thus $\dim(\La) = d-k $). Then $c_0, c_{k+1}, \ldots , c_{d-2}$ are    $(d-1-k)+1 = d-k$  ``independent'' critical points on $\La$     and arguing as in     Proposition \ref{prop:PCF} shows that \pcf maps are Zariski dense on $\La$.
    
     A ``dynamical 
André-Oort conjecture'' was proposed by Baker and DeMarco \cite{baker_demarco} which says precisely that   a subvariety 
$\La$ of dimension $q$ in the moduli space of rational maps of degree $d$ with marked critical points
 $\mathcal{M}_d^{\mathrm{cm}}$
 is special if and only if at most $q$ critical points are ``dynamically independent'' on $\La$ (the precise notion of dynamical dependence is slightly delicate to formalize). 

Some partial results towards this conjecture are known, including a complete proof for the space of cubic 
polynomials\footnote{Note added in April 2020: Favre and Gauthier have recently obtained a classification of
 special curves in spaces  of polynomials of arbitrary degree.}. Recall that 
$\mpoly_3^{\mathrm{cm}}$ is parameterized by $(c,a)\in \cd$.

\begin{thm}[Favre-Gauthier \cite{favre_gauthier2}, Ghioca-Ye \cite{ghioca_ye}]\label{thm:special curves}
An irreducible curve $C$ in the space $\mpoly_3^{\mathrm{cm}}$ is special if and only if one of the of the following holds:
\begin{itemize}
\item one of the two critical points is persistently preperiodic along $C$;
\item There is a persistent collision between the critical orbits, that is, there exists $(k, l)\in \nn^2$ such that 
$P_{c,a}^k (c_0) = P_{c,a}^l (c_1)$ on $C$;
\item $C$ is the curve of cubic polynomials $P_{c,a}$ commuting with $Q_c:z\mapsto -c+z$ (which is given by an explicit equation). 
 \end{itemize}
 \end{thm}
 
 The proof is too long to be described in these notes, let us just say a few words on how arithmetic equidistribution (again!) 
 enters into play. For $i=0, 1$, we define the function $G_i$    by $G_i (c,a)  = G_{c,a}(c_i)$, which again 
 admit adelic versions. Since $C$ admits infinitely many \pcf parameters, 
 Yuan's theorem applies to show  that they are equidistributed inside $C$. Now there are two parameter height functions on $C$, one associated to $c_0$ and the other one associated to $c_1$. 
Since \pcf parameters are of height 0 relative to both functions, we 
infer that the limiting measure must be proportional to both 
$\Delta(G_0\rest{C})$ and $\Delta(G_1\rest{C})$, thus $\Delta(G_0\rest{C}) = \alpha \Delta(G_1\rest{C})$
 for some $\alpha>0$.  This defines a first dynamical relation between $c_0$ and $c_1$, which after some work (which 
 involves in particular equidistribution at finite places),   is promoted to an analytic relation between the $c_i$, and finally to the 
 desired dynamical relation. 
 
We give  a more detailed argument in a particular case of Theorem \ref{thm:special curves}, which had previously
been  obtained  by Baker and DeMarco \cite{baker_demarco}. We let $\per_n(\kappa)$ be the algebraic curve in $\mpoly_3^\cm$ defined 
by the property that polynomials in $\per_n(\kappa)$ admit 
  a periodic point of exact period $n$ and multiplier $\kappa$. 
 
 \begin{thm} [Baker-DeMarco \cite{baker_demarco}]
 The curve $\per_1(\kappa)$ in $\mpoly_3^\cm$ 
 is special if and only if $\kappa=0$.
 \end{thm}  
 
 The same result actually holds for   $\per_n(\kappa)$ for every $n\geq 1$ by \cite[Thm B]{favre_gauthier2}.
 
\begin{proof}
The ``if'' implication is easy: $\per_1(0)$ is the set of cubic polynomials where some critical point is fixed. Consider an irreducible component of $\per_1(0)$, then one critical point, say $c_0$ is fixed. We claim that $c_1$ is not passive along that component. Indeed otherwise by Theorem \ref{thm:passive curve} it would be persistently preperiodic and we would get a curve of \pcf 
parameters. So there exists a parameter $\lo\in \per_1(0)$  at which $c_1$ is active, hence by Proposition \ref{prop:levin} we get 
 an infinite  sequence  $\la_n\to\lo$ for which $c_1(\la_n)$ is preperiodic. In particular  $\per_1(0)$  contains infinitely many \pcf 
 parameters and we are done. 
 
Before starting the proof of the direct implication,  observe that if $0< \abs{\kappa} <1$ by Theorem \ref{thm:attracting critical} a critical point must be attracted by the attracting fixed point so $\per_1(\kappa)$ is not special. A related  argument also applies for $\abs{\kappa} =1$, so the result is only interesting when $\abs{\kappa}>1$. The proof will be divided in several steps. We fix $\kappa\neq 0$. 
 
 {\bf Step 1: adapted parameterization.} We   first change coordinates in order 
  to find a parameterization of $C$ that is 
  convenient for the calculations to come. 
  First we can conjugate by 
   a translation so that the fixed point is $0$. The 
    general form of cubic polynomials with a fixed point at 0 of multiplier 
    $\kappa$ is $\kappa z+ az^2+ bz^3$ with $b\neq 0$. Now by a homothety we can adjust $b=1$ to get $\kappa z+ az^2+  z^3$. 
In this form the critical points are not marked. It turns out that a convenient parameterization of $C$  is given by
 $(f_s)_{s\in \cc^*}$ defined by 
$$ f_s(z) = \kappa \lrpar{z- \unsur{2}\lrpar{s+ \unsur{s}}z^2 + \unsur{3} z^3},$$
whose critical points are $s$ and $1/s$. Denote   $c^+(s) = s$ and $c^-(s) = s\inv$. The parameterization has 2-fold symmetry 
$s \leftrightarrow s\inv$. 

 {\bf Step 2: Green function estimates.} 
 We consider as usual the values of the dynamical Green function at critical points and define $G^+(s) = G_{f_s}(s)$ and 
$G^-(s) = G_{f_s}(s\inv)$. An elementary calculation shows that $s\mapsto f^n_s(s)$ is a polynomial in $s$, with 
the following leading coefficient
\begin{align*}
f_s^n(s)  &= \frac{\kappa}{3} \lrpar{\frac{\kappa}{3}}^3 \cdots \lrpar{\frac{\kappa}{3}}^{3^{n-2}} \lrpar{-\frac{\kappa}{6}} ^{3^{n-1}}
s^{3^n} + O\lrpar{s^{3^{n-1}}}\\
&= \lrpar{\frac{\kappa}{3}}^{\frac1{6} (3^{n-2} -1)} \lrpar{-\frac{\kappa}{6}} ^{3^{n-1}}
s^{3^n} + O\lrpar{s^{3^{n-1}}}.
\end{align*}
It follows that 
\begin{align}\label{eq:G+}
G^+(s)  = G_{f_s}(s)  &= \lim_{n\to\infty} \log^+\abs{f^n_s(s)}  \\
\notag &= \log\abs{s}+  \log \abs{\frac{\kappa}{6}}^{1/3} + \log \abs{\frac{\kappa}{3}}^{1/6}  + O(1)
\end{align}
when $s\to \infty$ (I am cheating here because the coefficient 
in $O(s^{3^{n-1}})$ is not uniform in $n$). On the other hand, since 
$f^n_s(s)$ is a polynomial, for $\abs{s}\leq 1$ we have $\abs{f^n_s(s)}\leq \abs{f^n_1(1)}$ so $G^+(s)\leq G^+(1)$ and 
it follows that $G^+$ is bounded near $s=0$. By symmetry $G^-$ is bounded near $\infty$ and tends to $+\infty$ when 
$s\to 0$. 

{\bf Step 3: bifurcations.} We now claim that $c^+$ and $c^-$ are not passive along $C$. Indeed as before otherwise they would be persistently preperiodic, 
contradicting the fact that $G^+$ and $G^-$ are unbounded. So we can define two bifurcation measures $\mu^+ = \Delta G^+$ 
and $\mu^- = \Delta G^-$ in $\cc^*$. Note that 
$\mu^+$ is a probability measure 
(because the coefficient of the log in \eqref{eq:G+} equals 1) and its support is bounded in $\cc$ (i.e. does not contain $\infty$) because the dynamical Green function is harmonic when it is positive. A similar description holds for $\mu^-$, whose support is 
away from 0. 

We saw in \S \ref{sec:equidist higher} that the bifurcation locus is the union of the activity loci of critical points. 
It follows that it   is contained in $\supp(\mu^+)\cup \supp(\mu^-)$  --this is actually an equality by a theorem of DeMarco \cite{demarco_tbif}. 

{\bf Step 4: equidistribution argument and conclusion.}  Assume now that   $\per_1(\kappa)$ contains infinitely many 
\pcf parameters. Note that the existence of a \pcf map in $\per_1(\kappa)$ implies that $\kappa\in \overline \Q$ so 
$\per_1(\kappa)$ is defined over some number field $K$.  
It can be shown that the functions $G^\pm$ 
(in their adelic version) satisfy the assumptions of the Yuan's arithmetic equidistribution theorem. Thus if 
$(s_k)_{k\geq 0}$ is any infinite sequence such that 
$c^+(s_k)$ is preperiodic, the uniform measures on the $\mathrm{Gal}(\overline K/K)$ conjugates of $s_k$ 
equidistribute towards $\mu^+$. Applying this 
fact along a sequence of \pcf parameters we conclude that $\mu^+ = \mu^-$. We want to derive a contradiction from this
equality. 

We have that  $\supp(\mu^+)  = \supp(\mu^-)$ so this set must be compact in $\cc^*$. The function $G^+- G^-$ is harmonic  on 
$\cc^*$  with $G^+(s) -G^-(s)  = \log \abs{s} +O(1)$ at $+\infty$ and 0. Therefore $G^+  -G^-   - \log \abs{\cdot}$ is harmonic 
and bounded on $\cc^*$ so it is constant, and we conclude that 
$G^+(s) - G^-(s) = \log\abs{s} +  C$ for some $C$. Now recall that the bifurcation locus $\bif$ is contained in 
$\supp(\mu^+)  \cup \supp(\mu^-) = \supp(\mu^+) $, so $\bif\subset \set{G^+=0}$. Likewise $\bif\subset \set{G^-=0}$ so 
 $\bif$ is contained in  $\set{G^+- G^-=0}$ which is a circle.
 On the other hand it follows from Theorem \ref{thm:shishikura} that  for the family $(f_s)_{s\in \cc^*}$, the Hausdorff dimension of $\bif$ is equal to 2. This contradiction finishes the proof. 
\end{proof}

 \bibliographystyle{plain}
\bibliography{refsgrenoble}

 \end{document}